\documentclass{amsart}

\usepackage{thmtools}
\usepackage{thm-restate}
\usepackage{float}
\usepackage[usenames,dvipsnames]{xcolor}
\usepackage{mathtools}
\usepackage{amssymb}
\usepackage{amsmath}
\usepackage{amsthm}
\usepackage{hyperref}
\usepackage{graphicx}
\usepackage{pinlabel}
\usepackage[rightcaption]{sidecap}
\usepackage{wrapfig}
\usepackage{caption}
\usepackage{subcaption}
\usepackage[all]{xy}
\usepackage[english]{babel}

\addto\extrasenglish{%
 }

\newtheorem{theorem}{Theorem}[section]
\theoremstyle{definition}
\newtheorem{definition}[theorem]{Definition}
\newtheorem{proposition}[theorem]{Proposition}
\newtheorem{example}[theorem]{Example}

\newtheorem{remark}[theorem]{Remark}

\newtheorem{lemma}[theorem]{Lemma}
\newtheorem{corollary}[theorem]{Corollary}
\newtheorem{conjecture}[theorem]{Conjecture}

\usepackage{anysize} \marginsize{1in}{1in}{1in}{1in}

\hypersetup{
    colorlinks=true,
    linkcolor=NavyBlue,
    citecolor=NavyBlue
}

\graphicspath{ {./Figures/} }
\numberwithin{equation}{section}

\begin{document}

\title{A note on some high-dimensional handlebodies}

\author{Geunyoung Kim}
\address{Department of Mathematics \& Statistics, McMaster University, Hamilton, Ontario, Canada}
\curraddr{}
\email{kimg68@mcmaster.ca}
\thanks{}
\date{}
\dedicatory{}

\begin{abstract}
For $k \geq 0$ and $n \geq 2k+1$, we show that every $n$-dimensional $k$-handlebody is the product of a $2k$-dimensional $k$-handlebody and the standard $(n-2k)$-ball. For $k \geq 2$ and $n \geq 2k$, we introduce $(n,k)$-Kirby diagrams for some $n$-dimensional $k$-handlebodies, where $(4,2)$-Kirby diagrams correspond to the original Kirby diagrams for $4$-dimensional $2$-handlebodies.
\end{abstract}

\maketitle
\addtocontents{toc}{\protect\setcounter{tocdepth}{1}}

\section{Introduction}

We work in the smooth category throughout. Let $B^m$ denote the standard $m$-ball, where $0\in B^n$ is the origin. An \textit{$n$-dimensional $k$-handle} is the product $B^{k}\times B^{n-k}$, and we call $k$ the \textit{index} of the handle. 

Let $X$ be an $n$-manifold with boundary, and let $\phi:S^{k-1}\times B^k\hookrightarrow \partial X$ be an embedding. Here, we refer to $S^{k-1}\times B^{n-k}\subset B^k\times B^{n-k}$ as the \textit{attaching region} of the handle and call $\phi$ the \textit{attaching map}. The quotient manifold \[X'=X\cup_\phi(B^k\times B^{n-k})=\frac{X\coprod (B^k\times B^{n-k})} {y\sim\phi(y)},\quad (y\in S^{k-1}\times B^k)\] is called the \textit{manifold obtained from $X$ by attaching a $k$-handle along $\phi$}. 

Note that the diffeomorphism type of $X'$ is determined by the isotopy class of $\phi.$ By the tubular neighborhood theorem \cite{guillemin2010differential}, an embedding $\phi:S^{k-1}\times B^{n-k}\hookrightarrow \partial X$ is uniquely determined up to isotopy by the attaching sphere $S=\phi(S^{k-1}\times \{0\})$ of the $k$-handle and a framing of $S$, which is an identification of the normal bundle $\nu S$ of $S$ with $S^{k-1}\times \mathbb{R}^{n-k}$.
After choosing a fixed reference framing, the set of framings can be canonically identified with $\pi_{k-1}(GL(n-k))\cong \pi_{k-1}(O(n-k))$ (see \cite{gompf20234}). 

An \textit{$n$-dimensional $k$-handlebody} $X$ is an $n$-manifold obtained from the standard $n$-ball $B^n$ by attaching handles up to index $k$, i.e., $X$ admits a handle decomposition \[B^n=X_0\subset X_1\subset \dots\subset X_k=X,\] where $X_i$ is obtained from $X_{i-1}$ by attaching $i$-handles. Note that an $n$-dimensional $k$-handlebody is not merely the union of a $0$-handle and $k$-handles. 

The following is our main theorem.

\begin{restatable*}{theorem}{productstructureonhandlebodytwo}\label{thm: product structure on handlebodies}
     Fix $k\geq0$ and $n\geq2k+1$. Let $X_k$ be an $n$-dimensional $k$-handlebody. Then there exists a $2k$-dimensional $k$-handlebody $Y_k$ such that $X_k$ is diffeomorphic to $Y_k\times B^{n-2k}$.
\end{restatable*}

\begin{remark}
    In \autoref{thm: product structure on handlebodies}, there may exist non-diffeomorphic $2k$-dimensional $k$-handlebodies $Y$ and $Y'$ such that $Y\times B^1\cong Y'\times B^1$, which implies that $Y\times B^{n-2k}\cong Y' \times B^{n-2k}$.
    \begin{enumerate}
        \item Let $Y$ be the once-punctured $S^k\times S^k$, and let $Y'$ be $(S^k\times B^k)\natural(S^k\times B^{k})$. Then $\partial Y\cong S^{2k-1}$ and $\partial Y'\cong (S^k\times S^{k-1})\#(S^k\times S^{k-1})$, so $Y$ and $Y'$ are not diffeomorphic. However, \[Y\times B^{1}\cong (S^k\times B^{k+1})\natural(S^k\times B^{k+1})\cong Y'\times B^{1}.\]
        \item Let $Y$ be a Mazur manifold, a contractible $4$-manifold that is not homeomorphic to $B^4$, with a handle decomposition consisting of a $0$-handle, a $1$-handle, and a $2$-handle, where the $1$- and $2$- handles algebraically cancel but not geometrically (see \cite{mazur1961note}). Note that $\partial Y$ is a non-simply connected homology $3$-sphere. Let $Y'$ be a $4$-manifold diffeomorphic to $B^4$ with a $0$-handle, and a $1$-handle, and a $2$-handle, where the $1$- and $2$-handles geometrically cancel. Then $Y\times B^1\cong B^5\cong Y'\times B^1$ but $Y$ and $Y'$ are not diffeomorphic.
    \end{enumerate}
\end{remark}

\begin{remark}
    In \autoref{sec: generalized Kirby diagrams}, we define $(n,k)$-Kirby diagrams for some $n$-dimensional $k$-handlebodies, generalizing the original Kirby diagrams for $4$-dimensional $2$-handlebodies. The original Kirby diagrams correspond to $(4,2)$-Kirby diagrams (see \cite{kirby1978calculus,kirby2006topology, gompf20234, akbulut20164} for more details on Kirby diagrams of $4$-manifolds. We use \autoref{thm: product structure on handlebodies} to perform generalized Kirby calculus on $(n,k)$-Kirby diagrams and to prove \autoref{thm: classification of handlebodies} below. Another application is relating $(n,2)$-Kirby diagrams to $(4,2)$-Kirby diagrams, allowing the $n$-manifolds represented by $(n,2)$-Kirby diagrams to be interpreted in terms of the $4$-manifolds represented by $(4,2)$-Kirby diagrams (see \autoref{rem: relation between (4,2) and (n,2)} and \autoref{A6}).
\end{remark}

\begin{restatable*}{theorem}{classificationofcertainhandlebodies}\label{thm: classification of handlebodies}
    Fix $k\geq2$ and $n\geq2k+1$. Let $X$ be an $n$-dimensional $k$-handlebody obtained from $S^{k-1}\times B^{n-k+1}$ by attaching $m$ $k$-handles, i.e., $X$ is an $n$-dimensional $k$-handlebody with a $0$-handle, a $(k-1)$-handle, and $m$ $k$-handles for some $m\geq 0$. Then $X$ is diffeomorphic to $M_{\mathcal{K}(p;a,b)}$ for some $(n,k)$-Kirby diagram $\mathcal{K}(p;a,b)$ in the right of \autoref{A5}.
\end{restatable*}

\begin{corollary}\label{cor: double of handle bodies}
    Fix $k\geq0$ and $n\geq2k+1$. Let $X_k$ be an $n$-dimensional $k$-handlebody. Then there exists an $(n-1)$-dimensional $k$-handlebody $Z_k$ such that \[X_k\cong Z_k\times B^{1}.\] In particular, the boundary $\partial X_k$ of $X_k$ is diffeomorphic to the double $DZ_k$ of $Z_k$.
    \begin{proof}
        By \autoref{thm: product structure on handlebodies}, we have \[X_k\cong Y_k\times B^{n-2k}\cong Y_k\times B^{n-2k-1}\times B^1.\] Let $Z_k=Y_k\times B^{n-2k-1}$, which is an $(n-1)$-dimensional $k$-handlebody. Then \[X_k\cong Z_k\times B^1.\] Thus, \[\partial X\cong\partial (Z_k\times B^1)\cong Z_k\cup_{id}\overline{Z_k}=DZ_k,\] where $id$ is the identity map on $\partial Z_k$.
    \end{proof}
\end{corollary}

Thus, for example, the boundary of a $5$-dimensional $2$-handlebody is the double of a $4$-dimensional $2$-handlebody.

\begin{corollary}
    Fix $k\geq0$ and $n\geq2k+2$. Let $X_k$ be an $n$-dimensional $k$-handlebody. Then there exists an $(n-2)$-dimensional $k$-handlebody $W_k$ such that \[X_k\cong W_k\times B^{2}.\] In particular, the boundary $\partial X_k$ of $X_k$ admits an open book decomposition, where the binding is $\partial W_k$ and each page is $W_k$.
\begin{proof}
    By \autoref{thm: product structure on handlebodies}, we have \[X_k\cong Y_k\times B^{n-2k}\cong Y_k\times B^{n-2k-2}\times B^2.\] Let $W_k=Y_k\times B^{n-2k-2}$, which is an $(n-2)$-dimensional $k$-handlebody. Then \[X_k\cong W_k\times B^2.\] Furthermore, \[\partial X_k\cong \partial(W_k\times B^2)=(\partial W_k\times B^2)\cup(W_k\times S^1).\] Thus, $\partial X_k$ admits a natural open book decomposition, where the binding is $\partial W_k$ and each page is $W_k$.
    \end{proof}
\end{corollary}

For example, the boundary of a $6$-dimensional $2$-handlebody admits an open book decomposition, where the binding is the boundary of a $4$-dimensional $2$-handlebody, and each page is the $4$-dimensional $2$-handlebody. 

For instance, let $X=M\times B^2$ be a $6$-dimensional $2$-handlebody, where $M$ is a Mazur manifold with a $0$-handle, a $1$-handle, and a $2$-handle. Since $X=M\times B^2\cong B^6$, the boundary $\partial X$ admits a natural open book decomposition of $S^5$, whose binding is the non-simply connected homology $3$-sphere $\partial M$, and each page is the contractible $4$-manifold $M$.

\begin{remark}
One may ask whether \autoref{cor: double of handle bodies} holds when $n\leq 2k$. The answer is no. 

Consider the once-punctured $S^{n-k}\times S^k$, denoted by $X$. This is an $n$-dimensional $k$-handlebody obtained from $B^n$ by attaching an $(n-k)$-handle along the unknotted $(n-k-1)$-sphere $A$ and a $k$-handle along the unknotted $(k-1)$-sphere $B$. Specifically, \[A\cup B=(S^{n-k-1}\times \{0\})\cup (\{0\}\times S^{k-1})\subset (S^{n-k-1}\times B^{k})\cup (B^{n-k}\times S^{k-1})=\partial (B^{n-k}\times B^{k})=\partial B^{n}=S^{n-1}.\] For example, when $n=4$ and $k=2$, $A\cup B$ is the Hopf link in $S^3$. 

Now, observe the homology group $H_{n-k}(X)$. It satisfies 
\begin{equation*}
H_{n-k}(X)\cong
    \begin{cases}
        \mathbb{Z}^2&\hspace{5mm}\text{if}\;\; n=2k\\
        \mathbb{Z} &\hspace{5mm}\text{if}\;\; n<2k.
    \end{cases}
\end{equation*}
Suppose that there exists an $(n-1)$-dimensional $k$-handlebody $Z$ such that $X\cong Z\times B^1$. This would imply that $Z$ is homotopy equivalent to $X$, meaning that $H_{n-k}(Z)$ is non-trivial. Consequently, $H_{n-k}(DZ)$ must also be non-trivial. However, this contradicts \[H_{n-k}(\partial X)=H_{n-k}(DZ),\] which must be trivial since $\partial X=S^{n-1}$.
\end{remark}

\begin{conjecture}\label{conj}
    Let $X$ and $X'$ be $(2k+1)$-dimensional $k$-handlebodies. If their boundaries $\partial X$ and $\partial X'$ are diffeomorphic, then $X$ and $X'$ are diffeomorphic.
\end{conjecture}

\begin{remark}\hfill
    \begin{enumerate}
        \item When $k=0,1$, \autoref{conj} is obviously true.
        \item When $k\geq3$, Lawson \cite{lawson1978open} showed that two $(2k+1)$-dimensional $k$-handlebodies $X$ and $X'$ with diffeomorphic boundaries are stably diffeomorphic. That is, there exists an integer $m\geq0$ such that \[X\natural \left(\natural^m\left(S^k\times B^{k+1}\right)\right)\cong X'\natural \left(\natural^m\left(S^k\times B^{k+1}\right)\right).\]
    \end{enumerate}
\end{remark}

\begin{remark}
    Let $X$ be a closed $(2k+1)$-manifold, and let $f:X\rightarrow [0,2k+1]$ be a self-indexing Morse function with a single index $0$ critical point and a single index $2k+1$ critical point. Then $X$ decomposes as \[X=X_1\cup_{\Sigma}X_2,\] where \[X_1=f^{-1}\left(\left[0,\frac{2k+1}{2}\right]\right)\quad \text{and}\quad X_2=f^{-1}\left(\left[\frac{2k+1}{2},2k+1\right]\right)\] are $(2k+1)$-dimensional $k$-handlebodies, and \[\Sigma=X_1\cap X_2=f^{-1}\left(\frac{2k+1}{2}\right).\] Note that $\Sigma$ is the double of a $2k$-dimensional $k$-handlebody by \autoref{cor: double of handle bodies}. It remains unclear whether $X_1$ and $X_2$ are necessarily diffeomorphic. 
    
    Lawson showed that every closed, orientable $(2k+1)$-manifold $X$ can be obtained by gluing two diffeomorphic $(2k+1)$-dimensional $k$-handlebodies when $k\geq 2$ (see \cite{lawson1978decomposing} for the case $k=2$ and \cite{lawson1978open} for $k\geq3$). However, this does not necessarily imply that $X_1$ and $X_2$ are always diffeomorphic. If \autoref{conj} is true, then Lawson's theorem follows immediately. 
    
    For $k=1$, $X_1$ and $X_2$ are always diffeomorphic, and the decomposition $X=X_1\cup X_2$ is known as a \textit{Heegaard splitting} of the closed $3$-manifold $X$. See \cite{reidemeister1933dreidimensionalen,singer1933three} for more details.
\end{remark}

\subsection*{Organization}In \autoref{sec: main theorems}, we prove \autoref{thm: product structure on handlebodies}. In \autoref{sec: generalized Kirby diagrams}, for $k\geq2$ and $n\geq2k$, we introduce $(n,k)$-Kirby diagrams for some $n$-dimensional $k$-handlebodies and demonstrate how these diagrams can be represented as $(4,2)$-Kirby diagrams. We then describe isotopies, handle slides, and creation/annihilation of canceling pairs within $(n,k)$-Kirby diagrams. These moves will play a crucial role in proving \autoref{thm: classification of handlebodies}. Additionally, we provide numerous examples of $(n,k)$-Kirby diagrams to illustrate these concepts.

\subsection*{Acknowledgements}The author would like to thank his doctoral advisor, David Gay, for valuable discussions as well as Maggie Miller and Mark Powell for their helpful comments. Much of this work was conducted at the University of Georgia and was partially supported by National Science Foundation grant DMS-2005554 ``Smooth $4$--Manifolds: $2$--, $3$--, $5$-- and $6$--Dimensional Perspectives''. The author also thanks Ian Hambleton and Patrick Naylor for frequently listening to his questions.

\section{Proof of \autoref{thm: product structure on handlebodies}}\label{sec: main theorems}
In order to prove \autoref{thm: product structure on handlebodies}, we use \autoref{lem: attaching map is isotopic to a product map} and an induction. Since \autoref{thm: product structure on handlebodies} is obvious when $k=0$, we fix $k\geq1$ and $n\geq2k+1$ . 

Consider an $n$-dimensional $(k-1)$-handlebody $X_{k-1}$ such that \[X_{k-1}=Y_{k-1}\times B^{n-2k}\] for some $2k$-dimensional $(k-1)$-handlebody $Y_{k-1}$. Let \[\Phi:S^{k-1}\times B^k\times B^{n-2k}\hookrightarrow\partial X_{k-1}\] be an attaching map of an $n$-dimensional $k$-handle, where the attaching region $S^{k-1}\times B^{n-k}$ of the $k$-handle is identified with $S^{k-1}\times B^k\times B^{n-2k}$. 

In \autoref{lem: attaching map is isotopic to a product map}, we show that there exists an attaching map of a $2k$-dimensional $k$-handle \[\phi:S^{k-1}\times B^{k}\hookrightarrow\partial Y_{k-1}\] such that \[X_k=X_{k-1}\cup_\Phi (B^k\times B^k\times  B^{n-2k})\cong(Y_{k-1}\cup_\phi (B^k\times B^{k}))\times B^{n-2k}=Y_k\times B^{n-2k}.\]

The key lemma used to prove \autoref{lem: attaching map is isotopic to a product map} is \autoref{lem: nicely embedded attaching spheres }, which states that the attaching sphere \[\Sigma=\Phi(S^{k-1}\times \{0\}\times\{0\})\] of the $k$-handle can be isotopic into a nice subspace \[\partial Y_{k-1}\times\{0\}\subset\partial X_{k-1},\] where \[\partial Y_{k-1}\times\{0\}\subset \partial Y_{k-1}\times B^{n-2k}\subset (\partial Y_{k-1}\times B^{n-2k}) \cup (Y_{k-1}\times S^{n-2k-1})\subset\partial(Y_{k-1}\times B^{n-2k})=\partial X_{k-1}.\] 

Consider \[\Sigma=\Phi(S^{k-1}\times\{0\}\times\{0\})=\Sigma_1\cup \Sigma_2\] as the union of two properly embedded $(k-1)$-manifolds $\Sigma_1$ and $\Sigma_2$, where \[\Sigma_1=\Sigma\cap(\partial Y_{k-1}\times B^{n-2k}) \quad\text{and}\quad \Sigma_2=\Sigma\cap(Y_{k-1}\times S^{n-2k-1}).\] 

In \autoref{lem: homotopic to the boundary of a handlebody}, we first show that $\Sigma_2$ is homotopic into \[\partial Y_{k-1}\times S^{n-2k-1}=\partial(Y_{k-1}\times S^{n-2k-1}),\] so we can assume that $\Sigma$ is homotopic into $\partial Y_{k-1}\times B^{n-2k}$ because \[\Sigma_1\subset \partial Y_{k-1}\times B^{n-2k} \quad\text{and}\quad \Sigma_2\subset \partial Y_{k-1}\times S^{n-2k-1}.\] Furthermore, using \autoref{thm: Whitney embedding theorems}, we show that $\Sigma$ is isotopic into \[\partial Y_{k-1}\times \{0\}\subset \partial Y_{k-1}\times B^{n-2k}.\]

\begin{theorem}[Whitney Embedding Theorems \cite{whitney1936differentiable,whitney1944self,ranicki2002algebraic}]\label{thm: Whitney embedding theorems}\hfill
    \begin{enumerate}
        \item Let $f: M\rightarrow N$ be a smooth map. If $\text{dim}(N)\geq 2\cdot \text{dim}(M)+1$, then $f$ is homotopic to an embedding $g:M\hookrightarrow N$.
        \item Let $f,g:M\hookrightarrow N$ be homotopic embeddings. If $\text{dim}(N)\geq 2\cdot \text{dim}(M)+2$, then $f$ and $g$ are isotopic.
    \end{enumerate}    
\end{theorem}

\begin{lemma}\label{lem: homotopic to the boundary of a handlebody}
    Fix $k\geq1$ and $n\geq 2k+1$. Let $Y_{k-1}$ be a $2k$-dimensional $(k-1)$-handlebody and $M$ be a $(k-1)$-manifold. If \[f:M\hookrightarrow Y_{k-1}\times S^{n-2k-1}\] is a proper embedding, then $f$ is homotopic to a map \[g:M\rightarrow \partial Y_{k-1}\times S^{n-2k-1}.\]
\end{lemma}
\begin{proof}
    Let $S(Y_{k-1})\subset Y_{k-1}$ be the canonical $(k-1)$-dimensional subcomplex (spine) of $Y_{k-1}$ obtained by collapsing the second factor of each handle to a point, i.e., $Y_{k-1}$ deformation retracts to $S(Y_{k-1})$. A spine $\mathcal{S}$ of $Y_{k-1}\times S^{n-2k-1}$ is given by \[\mathcal{S}=S(Y_{k-1}\times S^{n-2k-1})=S(Y_{k-1})\times S^{n-2k-1}.\] Since \[dim(\mathcal{S})=(k-1)+(n-2k-1)=n-k-2\] and \[\text{dim}(Y_{k-1}\times S^{n-k-1})-(\text{dim}(\mathcal{S})+\text{dim}(M))=(n-1)-((n-k-2)+(k-1))=2,\] we can assume that $f(M)$ and $\mathcal{S}$ do not intersect. Thus, $f$ is an embedding of $M$ in the complement of a regular neighborhood of $\mathcal{S}$. The complement is given by  \begin{align*}
    &\hspace{.5cm}(Y_{k-1}\times S^{n-k-1})\setminus \nu(\mathcal{S})\\ &= (Y_{k-1}\times S^{n-k-1})\setminus (\nu(S({Y_{k-1}}))\times S^{n-k-1})\quad(\text{since}\;\; \mathcal{S}=S(Y_{k-1})\times S^{n-k-1})) \\
     & =(Y_{k-1}\setminus \nu(S({Y_{k-1}})))\times S^{n-k-1}\quad (\text{since}\;\; A\times B\setminus C\times B=(A\setminus C)\times B)\\
     & =\partial Y_{k-1}\times[0,1]\times S^{n-k-1}\quad (\text{since}\;\; Y_{k-1}\setminus\nu(S(Y_{k-1}))=\partial Y_{k-1}\times[0,1]).
\end{align*}

Define the projection \[p: \partial Y_{k-1}\times [0,1]\times S^{n-k-1}\rightarrow \partial Y_{k-1}\times \{0\}\times S^{n-k-1}= \partial Y_{k-1}\times S^{n-k-1}\] by $p(y,t,z)=(y,0,z)=(y,z)$, where $\partial Y_{k-1}\times \{0\}\times S^{n-k-1}$ is identified with $\partial Y_{k-1}\times S^{n-k-1}$. Then the map \[g=p\circ f: M\rightarrow \partial Y_{k-1}\times S^{n-k-1}\] is homotopic to $f$, completing the proof.
\end{proof}

\begin{lemma}\label{lem: nicely embedded attaching spheres }
    Fix $k\geq1$ and $n\geq 2k+1$. Let $Y_{k-1}$ be a $2k$-dimensional $(k-1)$-handlebody and $M$ be a $(k-1)$-manifold. If \[f: M\hookrightarrow \partial (Y_{k-1}\times B^{n-2k})= (\partial Y_{k-1}\times B^{n-2k}) \cup (Y_{k-1}\times S^{n-2k-1})\] is an embedding, then $f$ is isotopic to an embedding \[j:M\hookrightarrow \partial Y_{k-1}\times \{0\}\subset\partial Y_{k-1}\times B^{n-2k}.\]
    \begin{proof}
    We may assume that $f$ is transverse to \[\partial Y_{k-1}\times S^{n-k-1}=(\partial Y_{k-1}\times B^{n-2k}) \cap (Y_{k-1}\times S^{n-2k-1}).\] Let \[Z=f^{-1}(f(M)\cap(Y_{k-1}\times S^{n-2k-1})).\] Then the restriction map \[f|_{Z}:Z\hookrightarrow Y_{k-1}\times S^{n-2k-1}\] is a proper embedding, so it is homotopic to a map \[g:Z\rightarrow \partial Y_{k-1}\times S^{n-2k-1}\] by \autoref{lem: homotopic to the boundary of a handlebody}. Therefore, we can assume that there exists a map \[h:M\rightarrow \partial Y_{k-1}\times B^{n-2k}\] such that $f$ is homotopic to $h$. Let \[p:\partial Y_{k-1}\times B^{n-2k}\rightarrow \partial Y_{k-1}\times\{0\}\] be the projection sending $(y,t)$ to $(y,0)$. Then $h$ is homotopic to \[p\circ h: M\rightarrow \partial Y_{k-1}\times\{0\}.\] Here, $p\circ h$ is homotopic to an embedding \[j:M\hookrightarrow \partial Y_{k-1}\times\{0\}\] by (1) in \autoref{thm: Whitney embedding theorems}, since \[ \text{dim}(\partial Y_{k-1}\times\{0\})=2k-1=2(k-1)+1=2\cdot \text{dim}(M)+1.\] We can consider \[j:M\rightarrow \partial Y_{k-1}\times\{0\}\subset \partial (Y_{k-1}\times B^{n-2k})\] as an embedding of $M$ in $\partial (Y_{k-1}\times B^{n-2k})$. Note that $f$ and $j$ are homotopic embeddings since \[f\sim h\sim p\circ h \sim j,\] where $a \sim b$ means that $a$ and $b$ are homotopic. Then $f$ and $j$ are isotopic by (2) in \autoref{thm: Whitney embedding theorems} since \[n\geq 2k+1 \;\text{and}\; \text{dim}(\partial (Y_{k-1}\times B^{n-2k}))=n-1\geq 2k=2(k-1)+2=2\cdot \text{dim}(M)+2.\]
    \end{proof}
\end{lemma}

\begin{remark}
In \autoref{lem: nicely embedded attaching spheres }, any two homotopic embeddings \[j, j': M\hookrightarrow \partial Y_{k-1}\times \{0\}\] in $\partial Y_{k-1}\times \{0\}$ are isotopic in $\partial(Y_{k-1}\times B^{n-2k})$ by $(2)$ in \autoref{thm: Whitney embedding theorems}. For example, consider the embeddings \[j,j': S^1\hookrightarrow \partial B^4\times \{0\}=S^3\times\{0\}\] such that \[j(S^1)=U\times\{0\}\;\; \text{and}\;\; j'(S^1)=T\times \{0\},\] where $U$ is the unknot and $T$ is the trefoil knot in $S^3$. Then $j$ and $j'$ are homotopic but not isotopic in $S^3\times \{0\}$. However, they become isotopic in $\partial (B^4\times B^1)=S^4$.
\end{remark}

We recall the following proposition to show that the framing of $\Phi(S^{k-1}\times \{0\}\times\{0\})$ is induced by a framing of $\phi(S^{k-1}\times \{0\})$, where $\Phi$ and $\phi$ are mentioned in the first paragraph of this section.

\begin{proposition}\label{pro: induced maps between homotopy groups of orthogonal groups}
    Let $i:O(l)\hookrightarrow O(l+1)$ be the natural inclusion of orthogonal groups. Then the induced homomorphism \[i_*:\pi_{m}(O(l))\rightarrow \pi_{m}(O(l+1))\] is an epimorphism if $l=m+1$ and an isomorphism if $l>m+1$.
\end{proposition}
\begin{proof}
    Consider the canonical Serre fibration \[O(l)\rightarrow O(l+1)\rightarrow O(l+1)/O(l)\cong S^l.\] This induces the long exact sequence: \[\cdots\rightarrow \pi_{m+1}(S^l)\longrightarrow \pi_m(O(l))\xrightarrow{i_{*}} \pi_m(O(l+1))\rightarrow \pi_m(S^l)\rightarrow \cdots.\]
    If $l=m+1$, then $\pi_{m+1}(S^l)\cong \mathbb{Z}$ and $\pi_m(S^l)=1$, so $i_*$ is an epimorphism. If $l>m+1$, then $\pi_{m+1}(S^l)=1$ and $\pi_m(S^l)=1$, so $i_*$ is an isomorphism.
\end{proof}

The following proposition follows immediately.

\begin{proposition}[Reformulation of \autoref{pro: induced maps between homotopy groups of orthogonal groups}]\label{lem:induced map}
    Fix $k\geq1$ and $n\geq2k+1$. Let $i:O(k)\hookrightarrow O(n-k)$ be the natural inclusion. Then the induced homomorphism \[i_*:\pi_{k-1}(O(k))\rightarrow \pi_{k-1}(O(n-k))\] is an epimorphism.
\begin{proof}
 Fix $k\geq1$. Let $P(n)$ denote the statement above. We use induction to prove that $P(n)$ holds for all $n\geq 2k+1$.\\
    \textit{Base Case}: $P(2k+1)$ is true by \autoref{pro: induced maps between homotopy groups of orthogonal groups}.\\
    \textit{Inductive Step}: Assume $P(m)$ is true.
    Let \[i=j\circ h:O(k)\xhookrightarrow{h} O(m-k)\xhookrightarrow{j} O(m+1-k)\] be the natural inclusion, where $h$ and $j$ are the natural inclusions. Then \[i_{*}=j_{*}\circ h_{*}\] is an epimorphism because $h_*$ is an epimorphism by the induction assumption and $j_{*}$ is an isomorphism by \autoref{pro: induced maps between homotopy groups of orthogonal groups}. Thus, the statement $P(m+1)$ also holds true.
\end{proof}

\end{proposition}

\begin{theorem}[Bott Periodicity Theorem \cite{bott1957stable}]\label{thm: Bott Periodicity theorem}
If $l\geq m+2$,

\begin{equation*}
\pi_m(O(l))\cong
    \begin{cases}
        \{0\}&\hspace{5mm}\text{if}\;\;m=2,4,5,6 \;(\text{mod 8}) \\
        \mathbb{Z}_2 &\hspace{5mm}\text{if}\;\;m=0,1 \;(\text{mod 8}) \\
        \mathbb{Z} & \hspace{5mm}\text{if}\;\;m=3,7 \;(\text{mod 8}).
    \end{cases}
\end{equation*}
\end{theorem}

In \autoref{thm: Bott Periodicity theorem}, let $m=k-1$ and $l=n-k$. The condition $l\geq m+2$ is equivalent to $n\geq 2k+1$.

\begin{theorem}[Restatement of \autoref{thm: Bott Periodicity theorem}]\label{thm:nice framings}
If $n\geq 2k+1$,

\begin{equation*}
G=\pi_{k-1}(O(n-k))\cong
    \begin{cases}
        \{0\}&\hspace{5mm}\text{if}\;\;k=3,5,6,7 \;(\text{mod 8}) \\
        \mathbb{Z}_2 & \hspace{5mm}\text{if}\;\;k=1,2 \;(\text{mod 8}) \\
        \mathbb{Z} & \hspace{5mm}\text{if}\;\;k=0,4 \;(\text{mod 8}).
    \end{cases}
\end{equation*}
\end{theorem}

\autoref{thm:nice framings} describes all possible framings of the attaching sphere of an $n$-dimensional $k$-handle when $n\geq 2k+1$. For example, a $5$-dimensional $2$-handle has $\mathbb{Z}_2$-framings, and a $7$-dimensional $3$-handles has a unique framing.

\begin{lemma}\label{lem:product structure on maps}
    Fix $k\geq1$ and $n\geq2k+1$. Let $Y$ be a $2k$-manifold with boundary. Suppose we have an embedding \[\phi:S^{k-1}\times B^{k}\hookrightarrow \partial Y.\] Define an embedding \[\Phi=\phi\times id_{B^{n-2k}}:S^{k-1}\times B^{k}\times B^{n-2k}\hookrightarrow \partial Y\times B^{n-2k}\subset \partial (Y\times B^{n-2k})\] by \[\Phi(x,y,t)=(\phi(x,y),t).\] Then \[(Y\times B^{n-2k})\cup_\Phi (B^k\times B^{k}\times B^{n-2k})\cong (Y\cup_\phi (B^k\times B^{k}))\times B^{n-2k}.\]
    
    \begin{proof} We begin with the $n$-manifold $(Y\times B^{n-2k})\cup_\Phi (B^k\times B^{k}\times B^{n-2k})$, which is obtained from the $n$-manifold $Y\times B^{n-2k}$ by attaching a $k$-handle along $\Phi$. Then
    \begin{align*}
    &\hspace{0.5cm} (Y\times B^{n-2k})\cup_\Phi (B^k\times B^{k}\times B^{n-2k})\\
    &=\frac{(Y\times B^{n-2k})\coprod (B^k\times B^{k}\times B^{n-2k})}{(x,y,t)\sim \Phi(x,y,t)}\;\;( \text{by definition of the handle attachment})\\
    &\cong \frac{(Y\times B^{n-2k})\coprod (B^k\times B^{k}\times B^{n-2k})}{(x,y,t)\sim (\phi(x,y),t)}\;\;(\text{since}\;\; \Phi(x,y,t)=(\phi(x,y),t))\\
    & \cong \frac{Y \coprod (B^k\times B^{k})}{(x,y)\sim \phi(x,y)}\times B^{n-2k}\;\;(\text{by definition of the quotient space})\\
     & \cong (Y\cup_\phi (B^k\times B^{k}))\times B^{n-2k}\;\;(\text{by definition of the handle attachment}).
\end{align*} Thus, \[(Y\times B^{n-2k})\cup_\Phi (B^k\times B^{k}\times B^{n-2k})\cong (Y\cup_\phi (B^k\times B^{k}))\times B^{n-2k}.\]
    \end{proof}
\end{lemma}

\begin{lemma}\label{lem: attaching map is isotopic to a product map}
    Fix $k\geq1$ and $n\geq2k+1$. Let $Y_{k-1}$ be a $2k$-dimensional $(k-1)$-handlebody, and $X_{k-1}=Y_{k-1}\times B^{n-2k}$ be an $n$-dimensional $k$-handlebody. Let \[\Phi:S^{k-1}\times B^k\times B^{n-2k}\hookrightarrow \partial X_{k-1}\] be an embedding. Then there exists an embedding \[\phi:S^{k-1}\times B^k\hookrightarrow \partial Y_{k-1}\] such that $\Phi$ is isotopic to $\phi\times id$. In particular, \[X_k=X_{k-1}\cup_{\Phi}(B^k\times B^k\times B^{n-k})\cong (Y_{k-1}\cup_{\phi}(B^k\times B^k))\times B^{n-2k}\cong Y_k\times B^{n-2k}.\]
\begin{proof} Let \[\Phi:S^{k-1}\times B^{k}\times B^{n-2k}\hookrightarrow \partial (Y_{k-1}\times B^{n-2k})\] be an embedding, and let \[X_k=X_{k-1}\cup_{\Phi}(B^k\times B^k\times B^{n-2k})\] be the $n$-dimensional $k$-handlebody obtained from $X_{k-1}$ by attaching a $k$-handle along $\Phi$.

By \autoref{lem: nicely embedded attaching spheres }, the attaching sphere $\Phi(S^{k-1}\times \{0\}\times\{0\})$ is isotopic to a sphere $F\times\{0\}\subset \partial Y_{k-1}\times\{0\}\subset \partial Y_k\times B^{n-2k}$ for some embedded $(k-1)$-sphere $F$ in $\partial Y_{k-1}$. Let $\nu(F)$ and $\nu(F\times\{0\})$ be closed regular neighborhoods in $\partial Y_{k-1}$ and $\partial Y_{k-1} \times B^{n-2k}\subset\partial(Y_{k-1}\times B^{n-2k})$, respectively. Then we can assume that $\nu(F\times\{0\})=\nu(F)\times B^{n-2k}\subset\partial(Y_{k-1}\times B^{n-2k})$.

By \autoref{lem:induced map}, a framing of $F\times \{0\}\subset\partial Y_{k-1}\times B^{n-2k}$ is induced by a framing of $F\subset\partial Y_{k-1}$. Therefore, there exists an embedding $\phi: S^{k-1}\times B^k\hookrightarrow  \partial Y_{k-1}$ such that 
    \begin{enumerate}
        \item $\phi(S^{k-1}\times \{0\})=F$,
        \item $\phi(S^{k-1}\times B^k)=\nu(F)$,
        \item $\phi\times id$ is isotopic to $\Phi$,
    \end{enumerate}
    where $\phi\times id: S^{k-1}\times B^k\times B^{n-2k}\hookrightarrow \partial Y_{k-1}\times B^{n-2k}\subset \partial (Y_{k-1}\times B^{n-2k})$ is defined by $(\phi\times id)(x,y,t)=(\phi(x,y),t)$. 
    
Furthermore, by \autoref{lem:product structure on maps}, we have $\begin{aligned}[t]
    X_k &= (Y_{k-1}\times B^{n-2k})\cup_\Phi (B^k\times B^k\times B^{n-2k})\\
     &\cong (Y_{k-1}\cup_\phi (B^k\times B^{k}))\times B^{n-2k}\\
     &\cong Y_k\times B^{n-2k},  \text{where $Y_k=Y_{k-1}\cup_\phi(B^k\times B^k)$.}
\end{aligned}$
\end{proof}
\end{lemma}

\begin{remark}\label{rmk: induced framings}
    In \autoref{lem: attaching map is isotopic to a product map}, assume that the framing $\Phi$ of $F\times\{0\}$ is induced by an element $\alpha\in\pi_{k-1}(O(n-k))$. Then the framing of $\phi$ of $F$ is induced by an element $\beta\in \pi_{k-1}(O(k))$ such that $i_{*}(\beta)=\alpha$, where $i_*:\pi_{k-1}(O(k))\rightarrow\pi_{k-1}(O(n-k))$ is the epimorphism in \autoref{pro: induced maps between homotopy groups of orthogonal groups}.
\end{remark}

\productstructureonhandlebodytwo

\begin{proof}
     Let $P(k)$ denote the statement above. We will use induction to prove that $P(k)$ holds for all $k\geq0$.\\
    \textit{Base Case}: $P(0)$ is true because \[X_0=B^n\cong B^{2k}\times B^{n-2k}=Y_0\times B^{n-2k},\] where $Y_0=B^{2k}$. \\
    \textit{Inductive Step}: Assume that $P(k)$ is true. Consider an $n$-dimensional $(k+1)$-handlebody \[X_{k+1}=X_{k}\cup\left( \coprod \left(B^{k+1}\times B^{k+1}\times B^{n-2(k+1)}\right)\right),\] which is obtained from an $n$-dimensional $k$-handlebody $X_k$ by attaching $(k+1)$-handles, where $n\geq2k+3=2(k+1)+1\geq 2k+1$. 
    
    By the inductive hypothesis, there exists a $2k$-dimensional $k$-handlebody $Y_k$ such that \[X_k\cong Y_k\times B^{n-2k}.\] We will show that there exists a $2(k+1)$-dimensional $(k+1)$-handlebody $Z_{k+1}$ such that \[X_{k+1}\cong Z_{k+1}\times B^{n-2(k+1)}.\] 
    
    Consider $X_k$ as $Y_k\times B^2\times B^{n-2(k+1)}$, and let $Z_k=Y_k\times B^2$ be a $2(k+1)$-dimensional $k$-handlebody. Note that \[X_{k+1}\cong (Z_{k}\times B^{n-2(k+1)})\cup \left(\coprod \left(B^{k+1}\times B^{k+1}\times B^{n-2(k+1)}\right)\right).\] Since $n\geq 2(k+1)+1$, by \autoref{lem: attaching map is isotopic to a product map}, there exists a $2(k+1)$-dimensional $(k+1)$-handlebody $Z_{k+1}$ such that \[X_{k+1}\cong Z_{k+1}\times B^{n-2(k+1)}.\] Thus, the statement $P(k+1)$ also holds true.
\end{proof}

\section{$(n,k)$-Kirby diagrams}\label{sec: generalized Kirby diagrams}
Fix $k\geq2$ and $n\geq2k$. We introduce $(n,k)$-Kirby diagrams of $n$-dimensional $k$-handlebodies $X$ such that \[X=\natural^a\left(S^{k-1}\times B^{n-k+1}\right)\cup_{\phi}\left(\coprod^b\left(B^k\times B^{n-k}\right)\right),\] where $\phi:\coprod^b(S^{k-1}\times B^{n-k})\hookrightarrow \#^a(S^{k-1}\times S^{n-k-2})$ is an embedding. Here, $\natural^a(S^{k-1}\times B^{n-k+1})$ is obtained from a $0$-handle $B^n$ by attaching $a$ $(k-1)$-handles trivially, and $X$ is obtained from $\natural^a(S^{k-1}\times B^{n-k+1})$ by attaching $b$ $k$-handles along $\phi$.

When $k=2$ and $n=4$, $X$ is a $4$-dimensional $2$-handlebody which can represented by a Kirby diagram $\mathcal{K}=L_1\cup L_2\subset S^3$, where $L_1$ is a dotted trivial $a$-component $1$-link, and $L_2$ is a framed $b$-component $1$-link representing $1$-handles and $2$-handles, respectively. For further details, see \cite{kirby1978calculus,kirby2006topology, gompf20234, akbulut20164}.

An \textit{$m$-link} $L\subset S^{n-1}$ is the image of an embedding $\phi:\coprod S^m\hookrightarrow S^{n-1}$ of a disjoint union of $m$-spheres in $S^{n-1}$, i.e., $L=\phi(\coprod S^m)$. We call $L$ an \textit{$m$-knot} if $L$ is the image of an embedding of a single $m$-sphere. An $m$-link $L$ is called \textit{trivial} if there exists an embedding $\Phi:\coprod B^{m+1}\hookrightarrow S^{n-1}$ of a disjoint union of $(m+1)$-balls in $S^{n-1}$ such that $\Phi(\coprod S^m)=L$. Here, we call $D_L=\phi(\coprod B^{m+1})$ the \textit{trivial $(m+1)$-balls} of the trivial $m$-link $L$.

\begin{definition}\label{def: generalized Kirby diagram}
    Fix $k\geq2$ and $n\geq2k$. An \textit{$(n,k)$-Kirby diagram} $\mathcal{K}=L_1\cup L_2\subset S^{n-1}=\partial B^n$ is the union of a dotted trivial $(n-k-1)$-link $L_1$ and a framed $(k-1)$-link $L_2$ such that $L_1\cap L_2=\emptyset$, i.e., there exists an embedding \[\phi:(\coprod S^{n-k-1})\coprod\left(\coprod \left(S^{k-1}\times B^{n-k}\right)\right)\hookrightarrow S^{n-1}\] such that $L_1=\phi(\coprod S^{n-k-1})$ and $L_2=\phi(\coprod (S^{k-1}\times \{0\}))$.
\end{definition}

\begin{remark}\label{rmk: (n,k)-Kirby diagrams}\hfill
\begin{enumerate}
    \item Akbulut \cite{Akbulut_1977} introduced the dotted trivial $1$-sphere notation to represent trivial $4$-dimensional $1$-handle attachment.  More generally, a dotted trivial $(i-j-2)$-sphere in $S^{i-1}=\partial B^i$ represents a trivial $i$-dimensional $j$-handle attachment; see \cite{akbulut20164, gay2021diffeomorphisms}. In particular, some figures in \cite{gay2021diffeomorphisms} illustrates Gay's dotted trivial $2$-sphere notation representing trivial $5$-dimensional $1$-handle attachment. In \autoref{def: generalized Kirby diagram}, we focus on the case that $i=n$ and $j=k-1$, meaning the dotted trivial $(n-k-1)$-link in $S^{n-1}$.
    \item Let $\mathcal{K}=L_1\cup L_2\subset S^{n-1}$ be an $(n,k)$-Kirby diagram, where $k\geq2$ and $n\geq2k+1$. Since $L_2$ itself is homotopic to the trivial $(k-1)$-link and \[\text{dim}(S^{n-1})=n-1\geq2k= 2(k-1)+2=2\cdot \text{dim}(L_2)+2,\] it follows from \autoref{thm: Whitney embedding theorems} that $L_2$ is isotopic to the trivial $(k-1)$-link in $S^{n-1}$. However this does not imply that $\mathcal{K}=L_1\cup L_2$ is trivial. Furthermore, $L_2$ has $G$-framings, where\begin{equation*}
G=\pi_{k-1}(O(n-k))\cong
    \begin{cases}
        \{0\}&\hspace{5mm}\text{if}\;\;k=3,5,6,7 \;\text{(mod 8)} \\
        \mathbb{Z}_2 & \hspace{5mm}\text{if}\;\;k=1,2 \;\text{(mod 8)} \\
        \mathbb{Z} & \hspace{5mm}\text{if}\;\;k=0,4 \;\text{(mod 8)}
    \end{cases}
\end{equation*} by \autoref{thm:nice framings}. Let $U\subset L_2\subset S^{n-1}$ be the trivial $(k-1)$-knot, which is a component of $L_2$. We define the \textit{$0$-framing} of $U$ as the canonical embedding $\phi:S^{k-1}\times B^{n-k}\hookrightarrow S^{n-1}$ such that
    \begin{enumerate}
        \item $\phi(S^{k-1}\times\{0\})=U$,
        \item $\phi(S^{k-1}\times B^{n-k})=\nu(U)$,
        \item $B^n\cup_{\phi}(B^k\times B^{n-k})\cong S^k\times B^{n-k}$.
    \end{enumerate} Then $m$-framing of $U$ can be canonically defined by an element $m\in G$. 
    
    \item For a $(2k,k)$-Kirby diagram $\mathcal{K}=L_1\cup L_2\subset S^{2k-1}$, $L_2$ can be a non-trivial link and has $\pi_{k-1}(O(k))$-framings. Therefore, $(n,k)$-Kirby diagrams (where $n\geq2k+1$) are generally simpler than $(2k,k)$-Kirby diagrams. For example, let $\mathcal{K}=L_1\cup L_2\subset S^3$ be a $(4,2)$-Kirby diagram. Then $L_2$ can be linked in $S^3$, and each component of $L_2$ has $\mathbb{Z}$-framings. Let $\mathcal{K'}=L_1'\cup L_2'\subset S^4$ be a $(5,2)$-Kirby diagram. Then $L_2'$ is always trivial in $S^4$ and can be isotoped to the equator $S^3\subset S^4$. Also, each component has $\mathbb{Z}_2$-framings. Later, we describe $(n,k)$-Kirby diagrams using $(4,2)$-Kirby diagrams when $n\geq2k+1$. In other words, given an $(n,k)$-Kirby diagram $\mathcal{K}$ in $S^{n-1}$, there exists a well-chosen $S^3\subset S^{n-1}$ such that $\mathcal{K}\cap S^3$ is a $(4,2)$-Kirby diagram in $S^3$.
    
\end{enumerate}
\end{remark}

We now construct the $n$-dimensional $k$-handlebody $M_{\mathcal{K}}$ from an $(n,k)$-Kirby diagram $\mathcal{K}$.

\begin{remark}
    Let $\mathcal{K}=L_1\cup L_2\subset S^{n-1}=\partial B^n$ be an $(n,k)$-Kirby diagram. Let $D_{L_1}\subset S^{n-1}$ be the trivial $(n-k)$-balls with boundary $\partial D_{L_1}=L_1$. Define $D_{L_1}'\subset B^n$ as the properly embedded $(n-k)$-balls obtained by pushing the interior of $D_{L_1}$ into $B^n$. 
    
    The closure of the complement $\overline{B^n\setminus \nu(D_{L_1}')}$, where $\nu(D_{L_1}')$ is a regular neighborhood of $D_{L_1}'$, is diffeomorphic to $\natural^{|L_1|}(S^{k-1}\times B^{n-k+1})$. Here, $\overline{B^n\setminus \nu(D_{L_1}')}$ consists of a $0$-handle and trivial $|L_1|$ $(k-1)$-handles. The link $L_2$ can be considered as a framed $(k-1)$-link in \[S^{n-1}\setminus \nu(L_1)\subset\partial(\overline{B^n\setminus \nu(D_{L_1}')})\cong \#^{|L_1|}(S^{k-1}\times S^{n-k}).\] We note that any $(k-1)$-link in $\#^{|L_1|}(S^{k-1}\times S^{n-k})$ can be isotoped into $S^{n-1}\setminus \nu(L_1)$ because $n\geq2k$.
\end{remark}
    
\begin{definition}
    We define $M_{\mathcal{K}}=\overline{B^n\setminus \nu(D_{L_1}')}\cup_{L_2} \text{$k$-handles}$ as the $n$-dimensional $k$-handlebody obtained from $\overline{B^n\setminus \nu(D_{L_1}')}$ by attaching $k$-handles along $L_2$. We call $M_{\mathcal{K}}$ an $\textit{$n$-dimensional $k$-handlebody}$ induced by $\mathcal{K}$.
\end{definition}

Now, consider an $(n,k)$-Kirby diagram $\mathcal{K}=L_1\cup L_2\subset S^{n-1}$ and ignore the framing of the $(k-1)$-link $L_2$. We still refer to $\mathcal{K}$ as an \textit{$(n,k)$-Kirby diagram}, but the context will make it clear whether $L_2$ is considered as a framed link or not. 

Next, we identify $S^3$ and $S^{n-1}$ with their respective one-point compactifications:

\[S^3=\mathbb{R}^3\cup\{\infty\} \quad\text{and}\quad S^{n-1}=(\mathbb{R}^{n-k-2}\times\mathbb{R}^3\times \mathbb{R}^{k-2})\cup\{\infty\}
     =\mathbb{R}^{n-1}\cup \{\infty\}.\]

\begin{theorem}\label{thm: $(n,k)$-Kirby diagram induced by $(4,2)$-Kirby diagram}
    Fix $k\geq2$ and $n\geq2k$ with $(n,k)\neq(4,2)$. Let $\mathcal{K}=L_1\cup L_2\subset\mathbb{R}^3$ be a $(4,2)$-Kirby diagram. Then there exists an $(n,k)$-Kirby diagram $\tilde{\mathcal{K}}=\tilde{L_1}\cup\tilde{L_2}\subset \mathbb{R}^{n-k-2}\times\mathbb{R}^3\times\mathbb{R}^{k-2}$ such that \[\tilde{\mathcal{K}}\cap(\{0\}\times\mathbb{R}^3\times\{0\})=\{0\}\times\mathcal{K}\times\{0\}\subset\mathbb{R}^{n-k-2}\times\mathbb{R}^3\times\mathbb{R}^{k-2}.\]

\begin{proof}
Let $\mathcal{K}=L_1\cup L_2\subset\mathbb{R}^3$ be a $(4,2)$-Kirby diagram. Let $D_{L_1}\subset \mathbb{R}^3$ be the trivial $2$-disks with $\partial D_{L_1}=L_1$. Define \[\tilde{L_1}=\partial(B^{n-k-2}\times D_{L_1}\times\{0\})\subset \mathbb{R}^{n-k-2}\times\mathbb{R}^3\times\mathbb{R}^{k-2},\] where $B^m\subset\mathbb{R}^m$ be the standard unit $m$-ball. Then $\tilde{L_1}$ is the trivial $(n-k-1)$-link in $\mathbb{R}^{n-k-2}\times\mathbb{R}^3\times\mathbb{R}^{k-2}$. 

Now, consider the $1$-link $\{0\}\times L_2\subset \mathbb{R}^{n-k-2}\times\mathbb{R}^3$. Let $D_{\{0\}\times L_2}\subset \mathbb{R}^{n-k-2}\times\mathbb{R}^3$ be the trivial $2$-disks with $\partial D_{\{0\}\times L_2}=\{0\}\times L_2$. Define \[\tilde{L_2}=\partial (D_{\{0\}\times L_2}\times B^{k-2})\subset \mathbb{R}^{n-k-2}\times\mathbb{R}^3\times\mathbb{R}^{k-2}.\] Then $\tilde{L_2}$ is the trivial $(k-1)$-link in $\mathbb{R}^{n-k-2}\times\mathbb{R}^3\times\mathbb{R}^{k-2}$. Note that \[\tilde{L_1}=(S^{n-k-3}\times D_{L_1}\times\{0\})\cup (B^{n-k-2}\times L_1 \times \{0\})\subset \mathbb{R}^{n-k-2}\times\mathbb{R}^3\times\mathbb{R}^{k-2}\] and \[\tilde{L_2}=(\{0\}\times L_{2}\times B^{k-2})\cup (D_{\{0\}\times L_2}\times S^{k-3})\subset\mathbb{R}^{n-k-2}\times\mathbb{R}^3\times\mathbb{R}^{k-2}.\] Let $\tilde{\mathcal{K}}=\tilde{L_1}\cup\tilde{L_2}\subset\mathbb{R}^{n-k-2}\times\mathbb{R}^3\times\mathbb{R}^{k-2}$. Then $\tilde{\mathcal{K}}$ is an $(n,k)$-Kirby diagram, and we have

\[\begin{aligned}[t]
    \tilde{\mathcal{K}}\cap(\{0\}\times\mathbb{R}^3\times\{0\}) &= (\tilde{L_1}\cup \tilde{L_2})\cap(\{0\}\times\mathbb{R}^3\times\{0\}) \\
     &= (\tilde{L_1}\cap(\{0\}\times\mathbb{R}^3\times\{0\}))\cup(\tilde{L_2}\cap(\{0\}\times\mathbb{R}^3\times\{0\}))\\
     &= (\{0\}\times L_1\times\{0\})\cup(\{0\}\times L_2\times\{0\})\\
     &= \{0\}\times (L_1\cup L_2)\times\{0\}\\
     &=\{0\}\times\mathcal{K}\times\{0\}\subset\mathbb{R}^{n-k-2}\times\mathbb{R}^3\times\mathbb{R}^{k-2}.
\end{aligned}\]
\end{proof}
\end{theorem}

\begin{remark}
    In \autoref{thm: $(n,k)$-Kirby diagram induced by $(4,2)$-Kirby diagram}, the link $\tilde{L_1}$ is the trivial $(n-k-1)$-link, and $\tilde{L_2}$ is the trivial $(k-1)$-link. This is because $\tilde{L_1}$ bounds the trivial $(n-k)$-balls $B^{n-k-2}\times D_{L_1}\times\{0\}$, and $\tilde{L_2}$ bounds the trivial $k$-balls $D_{\{0\}\times L_2}\times B^{k-2}$. However, this does not necessarily imply that $\tilde{K}=\tilde{L_1}\cup\tilde{L_2}$ is trivial.
\end{remark}

\begin{definition}
    In \autoref{thm: $(n,k)$-Kirby diagram induced by $(4,2)$-Kirby diagram}, we call $\tilde{L_1}$ the \textit{$(n-k-1)$-link induced by $L_1$}, $\tilde{L_2}$ the \textit{$(k-1)$-link induced by $L_2$}, and $\tilde{\mathcal{K}}$ the \textit{$(n,k)$-Kirby diagram induced by $\mathcal{K}$}.
\end{definition}

\begin{theorem}\label{thm: simplified Kirby diagrams}
    Fix $k\geq2$ and $n\geq2k+1$. Any $(n,k)$-Kirby diagram is isotopic to an $(n,k)$-Kirby diagram $\tilde{\mathcal{K}}$ induced by some $(4,2)$-Kirby diagram $\mathcal{K}$.

\end{theorem}
\begin{proof}
Let $\mathcal{J}=J_1\cup J_2\subset \mathbb{R}^{n-k-2}\times\mathbb{R}^3\times\mathbb{R}^{k-2}$ be an $(n,k)$-Kirby diagram. Since $J_1$ is the trivial $(n-k-1)$-link, we can assume that $J_1=\tilde{L_1}$, where $\tilde{L_1}\subset\mathbb{R}^{n-k-1}\times\mathbb{R}^3\times\mathbb{R}^{k-2}$ is the trivial $(n-k-1)$-link induced by the trivial $1$-link $L_1\subset\mathbb{R}^3$. 

We need to show that the $(k-1)$-link $J_2\subset\mathbb{R}^{n-k-2}\times\mathbb{R}^3\times\mathbb{R}^3\setminus\tilde{L_1}$ is isotopic to a $(k-1)$-link $\tilde{L_2}$ induced by some $1$-link $L_1\subset \mathbb{R}^3\setminus L_1$. 

Consider the map \[f:\pi_1(\mathbb{R}^3\setminus L_1)\rightarrow\pi_{k-1}(\mathbb{R}^{n-k-2}\times\mathbb{R}^3\times\mathbb{R}^{k-2}\setminus\tilde{L_1})\] defined by $f([\alpha])=[\tilde{\alpha}]$, where $\alpha$ is a $1$-knot in $\mathbb{R}^3\setminus L_1$. Here, any element of $\pi_1(\mathbb{R}^3\setminus L_1)$ is realized by a $1$-knot in $\mathbb{R}^3\setminus L_1$. We note that \[\pi_1(\mathbb{R}^3\setminus L_1)\cong *_{|L_1|}\mathbb{Z} \quad\text{and}\quad \pi_{k-1}(\mathbb{R}^{n-k-2}\times\mathbb{R}^3\times\mathbb{R}^{k-2}\setminus \tilde{L_1})\cong\bigoplus_{|L_1|}\mathbb{Z}.\] By the construction of the induced $(k-1)$-knot $\tilde{\alpha}$, $f$ is surjective. Let \[J_2=\beta_1\cup\cdots\cup\beta_m\subset\mathbb{R}^{n-k-2}\times\mathbb{R}^3\times\mathbb{R}^{k-2}\setminus\tilde{L_1}.\] Then for every $(k-1)$-knot $\beta_i$, there exists a $1$-knot $\alpha_i\subset\mathbb{R}^3\setminus L_1$ such that $f([\alpha_i])=[\tilde{\alpha_i}]=[\beta_i]$. Let \[L_2=\alpha_1\cup\cdots\cup\alpha_m\subset\mathbb{R}^3\setminus L_1.\] We can assume that $L_2$ is a $1$-link in $\mathbb{R}^3\setminus L_1$ after a small perturbation if necessary. Then the $(k-1)$-link $\tilde{L_2}$ is homotopic to $J_2$ in $\mathbb{R}^{n-k-2}\times\mathbb{R}^3\times\mathbb{R}^{k-2}\setminus\tilde{L_2}$. 

Since homotopy implies isotopy for $(k-1)$-manifolds in $(n-1)$-manifolds when $k\geq2$ and $n\geq2k+1$ (see \autoref{thm: Whitney embedding theorems}), we conclude that $J_2$ is isotopic to $\tilde{L_2}$ in $\mathbb{R}^{n-k-2}\times\mathbb{R}^3\times\mathbb{R}^{k-2}\setminus\tilde{L_2}$. Thus, $\mathcal{J}$ is isotopic to the $(n,k)$-Kirby diagram $\tilde{\mathcal{K}}=\tilde{L_1}\cup\tilde{L_2}$ induced by the $(4,2)$-Kirby diagram $\mathcal{K}=L_1\cup L_2$.
\end{proof}

\begin{remark}
    In the proof of \autoref{thm: simplified Kirby diagrams}, there are many different ways of perturbation of $L_2$, meaning that non-isotopic $(4,2)$-Kirby diagrams $\mathcal{K}$ and $\mathcal{K'}$ may induce isotopic $(n,k)$-Kirby diagrams $\tilde{\mathcal{K}}$ and $\tilde{\mathcal{K}'}$ are isotopic to $\mathcal{J}$. However, in \autoref{thm: crossing change} and \autoref{thm: other-move}, we will show that if $\mathcal{K}$ and $\mathcal{K}'$ are related by a crossing change and a $\ddagger$-move, then their induced $(n,k)$-Kirby diagrams $\tilde{\mathcal{K}}$ and $\tilde{\mathcal{K}'}$ are isotopic (see \autoref{A1}).
\end{remark}

From now on, we will represent $(n,k)$-Kirby diagrams using their corresponding $(4,2)$-Kirby diagrams, where $k\geq2$ and $n\geq2k+1$.

\begin{theorem}\label{thm: crossing change}
    Fix $k\geq2$ and $n\geq2k+1$. Let $\mathcal{K}=L_1\cup L_2\subset \mathbb{R}^3$ be a $(4,2)$-Kirby diagram, and let $\mathcal{K}'=L_1\cup L_2'\subset \mathbb{R}^3$ be another $(4,2)$-Kirby diagram, where $L_2'$ is obtained from $L_2$ by a crossing change, as shwon in \autoref{A1}. Then the induced $(n,k)$-Kirby diagrams $\tilde{\mathcal{K}}$ and $\tilde{\mathcal{K}'}$ are isotopic.
\end{theorem}
    \begin{proof}
        We use the same notation as in \autoref{thm: $(n,k)$-Kirby diagram induced by $(4,2)$-Kirby diagram}. Since $L_2$ and $L_2'$ are related by a crossing change, they are homotopic in $\mathbb{R}^3\setminus L_1$. Clearly, the two $1$-links $\{0\}\times L_2\times\{0\}$ and $\{0\}\times L_2'\times\{0\}$ are homotopic in $\mathbb{R}^{n-k-2}\times\mathbb{R}^3\times\{0\}\setminus \tilde {L_1}$. Since homotopy implies isotopy for $1$-manifolds in $(n-k+1)$-manifolds (see \autoref{thm: Whitney embedding theorems}), these two links are isotopic in $\mathbb{R}^{n-k-2}\times\mathbb{R}^3\times\{0\}\setminus \tilde {L_1}$. 
        
        More precisely, we can push the overstrand in the top left of \autoref{A1} to a different level set $\{t\}\times\mathbb{R}^{3}\times\{0\}\setminus \tilde{L_1}$, where $t\neq 0\in \mathbb{R}^{n-k-2}$, perform a small isotopy in that level set, and then pull it back to the lowerstrand in the top right of \autoref{A1}. Therefore, the induced $(k-1)$-links $\tilde{L_2}$ and $\tilde{L_2'}$ are isotopic in $\mathbb{R}^{n-k-2}\times\mathbb{R}^3\times\mathbb{R}^{k-2}\setminus\tilde{L_1}$, implying that $\tilde{\mathcal{K}}$ and $\tilde{\mathcal{K}'}$ are isotopic.
    \end{proof}

\begin{theorem}\label{thm: other-move}
    Fix $k\geq3$ and $n\geq2k+1$. Let $\mathcal{K}=L_1\cup L_2\subset \mathbb{R}^3$ be a $(4,2)$-Kirby diagram, and let $\mathcal{K}'=L_1\cup L_2'\subset \mathbb{R}^3$ be another $(4,2)$-Kirby diagram, where $L_2'$ is obtained from $L_2$ by a $\ddagger$-move, as shown in \autoref{A1}. Then the induced $(n,k)$-Kirby diagrams $\tilde{\mathcal{K}}$ and $\tilde{\mathcal{K}'}$ are isotopic.
\end{theorem}
    \begin{proof}
        Again, we use the same notation as in \autoref{thm: $(n,k)$-Kirby diagram induced by $(4,2)$-Kirby diagram}. We can consider $L_2$ as the result of performing surgery on $L_2''$ along a $2$-dimensional $1$-handle $h_C\subset\mathbb{R}^3$, whose core is $C$ in \autoref{A1}. Similarly, $L_2'$ is obtained from $L_2''$ by surgery along another $2$-dimensional $1$-handle $h_D\subset\mathbb{R}^3$, whose core is $D$ in \autoref{A1}. Although $h_C$ and $h_D$ are not explicitly drawn, we can imagine the trivial bands in the figure. 
        
        Let $H_C=\{0\}\times h_c\times B^{k-2}$ and $H_D=\{0\}\times h_D\times B^{k-2}$ be $k$-dimensional $1$-handles in $\{0\}\times\mathbb{R}^3\times\mathbb{R}^{k-2}\setminus\tilde{L_1}$, obtained from the $2$-dimensional $1$-handles $h_C$ and $h_D$ in $\mathbb{R}^3\setminus L_1$ by thickening them with $B^{k-2}\subset\mathbb{R}^{k-2}$, respectively. Then the induced $(k-1)$-links $\tilde{L_2}$ and $\tilde{L_2}'$ are obtained from $\tilde{L_2''}$ by surgery along $H_C$ and $H_D$, respectively. 
        
        If there exists an isotopy from the core $\{0\}\times C\times\{0\}$ to $\{0\}\times D\times\{0\}$ in $\{0\}\times\mathbb{R}^3\times\mathbb{R}^{k-2}\setminus \tilde{L_1}$, then it induces an isotopy from the $1$-handle $H_C$ to $H_D$ in $\{0\}\times\mathbb{R}^3\times\mathbb{R}^{k-2}\setminus \tilde{L_1}$, so the results of surgery, $\tilde{L_2}$ and $\tilde{L_2}'$ are isotopic. 
        
        We can push $\{0\}\times C\times\{0\}$ in $\{0\}\times\mathbb{R}^3\times\{0\}$ into a different level set $\{0\}\times\mathbb{R}^3\times\{t\}\subset\{0\}\times\mathbb{R}^3\times\mathbb{R}^{k-2}$, where $t\neq 0\in \mathbb{R}^{k-2}$, perform a small isotopy in that level set, and pull it back to $\{0\}\times D\times \{0\}$. Therefore, $\tilde{L_2}$ and $\tilde{L_2}'$ are isotopic in $\mathbb{R}^{n-k-2}\times\mathbb{R}^3\times\mathbb{R}^{k-2}\setminus\tilde{L_1}$, implying that $\tilde{\mathcal{K}}$ and $\tilde{\mathcal{K}'}$ are isotopic.
    \end{proof}

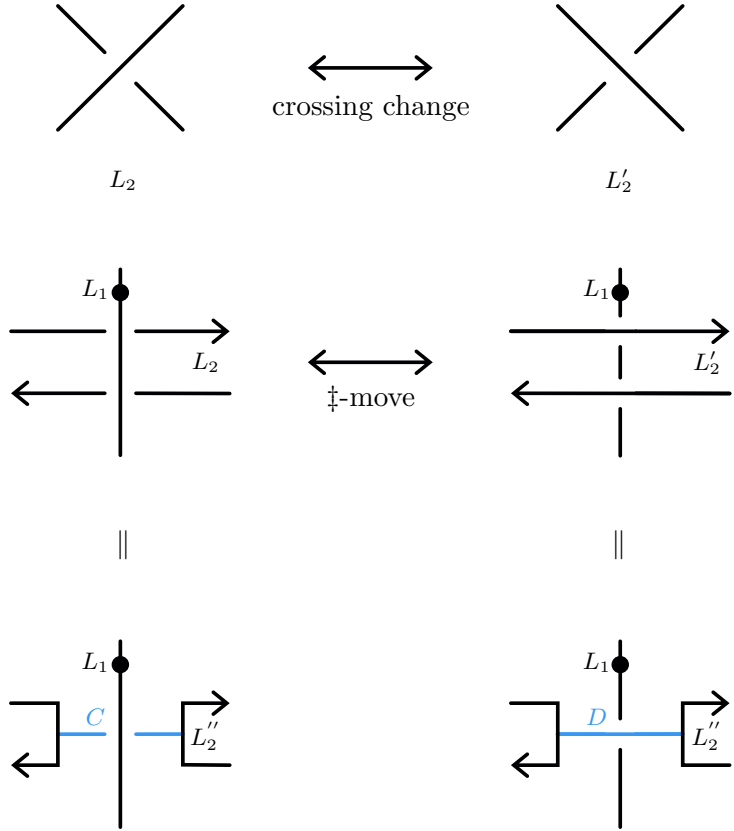
\begin{figure}[ht!]
 \labellist
 \small\hair 2pt
 
 \pinlabel {\large{crossing change}}  at 230 420
 \pinlabel {$L_2$} at 95 380
 \pinlabel {$L_2'$} at 365 380

 \pinlabel {\large{$\ddagger$-move}}  at 230 260
 \pinlabel {$L_1$} at 80 320
 \pinlabel {$L_2$} at 140 280
 \pinlabel {$L_1$} at 353 320
 \pinlabel {$L_2'$} at 413 280
 \pinlabel {$L_1$} at 80 115
 \pinlabel {$L_2^{''}$} at 140 75
 \pinlabel {$L_1$} at 353 115
 \pinlabel {$L_2^{''}$} at 413 75
 \pinlabel {\textcolor[RGB]{56,151,241}{$C$}}  at 80 85
 \pinlabel {\textcolor[RGB]{56,151,241}{$D$}}  at 353 85

 \pinlabel {\large{$\parallel$}}  at 95 180
 \pinlabel {\large{$\parallel$}}  at 365 180
 \endlabellist
 \centering
 
 \includegraphics[width=.7\textwidth]{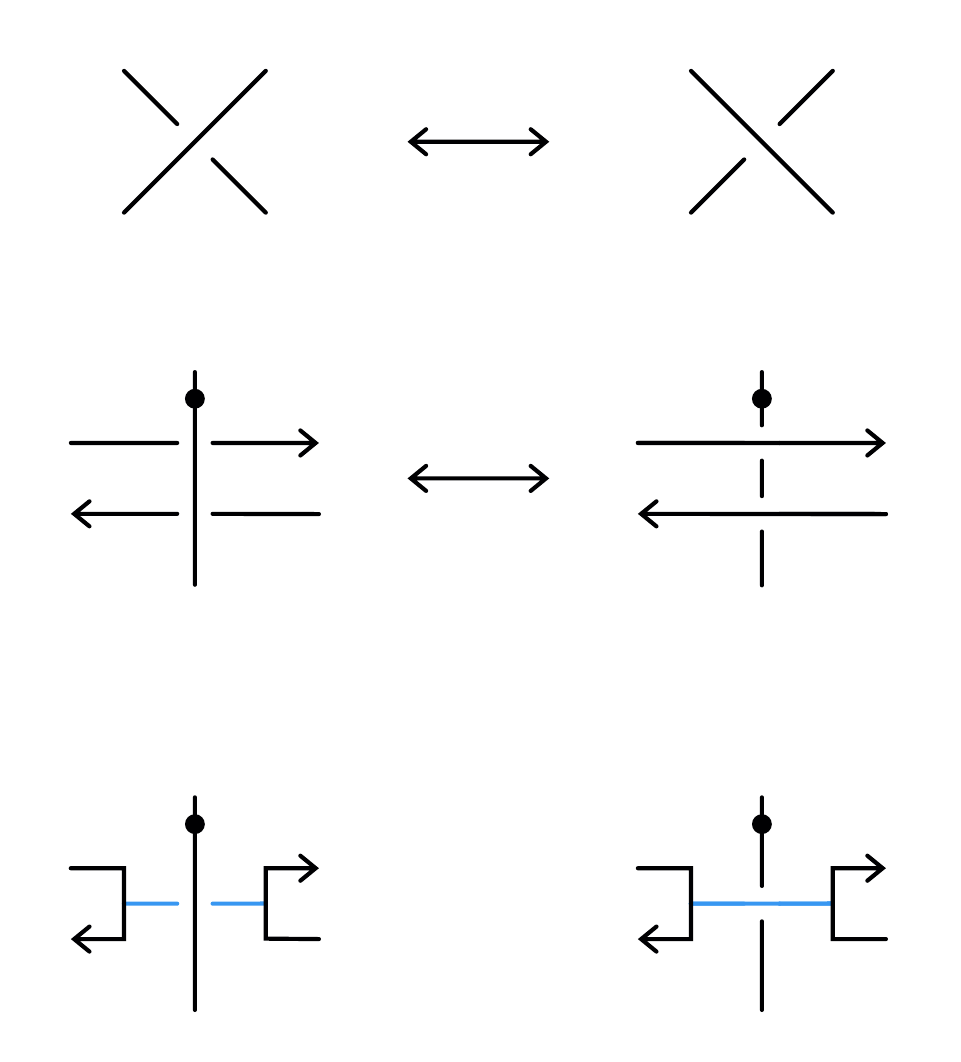}
 \caption{\textbf{First row}: A crossing change from $L_2$ to $L_2'$ induces an isotopy from the $(k-1)$-link $\tilde{L_2}$ to $\tilde{L_2'}$ in $\mathbb{R}^{n-k-2}\times\mathbb{R}^3\times\mathbb{R}^{k-2}\setminus \tilde{L_1}$, where $k\geq2$. \textbf{Second row}: A $\ddagger$-move from $L_2$ to $L_2'$ induces an isotopy from the $(k-1)$-link $\tilde{L_2}$ to $\tilde{L_2'}$ in $\mathbb{R}^{k-2}\times\mathbb{R}^3\times\mathbb{R}^{n-k-2}$. Note that two parallel strands with opposite orientations belong to the same component of $L_2$. \textbf{Third row}: $L_2$ is obtained from $L_2''$ by surgery along a $2$-dimensional $1$-handle $h_C$, whose core is $C$. Similarly, $L_2'$ is obtained from $L_2''$ by surgery along a $2$-dimensional $1$-handle $h_D$, whose core is $D$.
 \label{A1}}
\end{figure}

\begin{example}
    See \autoref{A2}.
    \begin{enumerate}
        \item Consider $\mathcal{K}_1$ and $\mathcal{K}_2$. 
        \begin{enumerate}
            \item When $k=2$ and $n=4$, $\mathcal{K}_1$ represents a Mazur manifold, which is contractible but not diffeomorphic to $B^4$ \cite{mazur1961note}, and $\mathcal{K}_2$ represents $B^4$. Here, $\mathcal{K}_2$ is a $(1,2)$-cancelling pair.
            \item When $k\geq2$ and $n\geq 2k+1$, $\mathcal{K}_1$ and $\mathcal{K}_2$ are related by a crossing change (when $k\geq2$) in the first row of \autoref{A1}, or a $\ddagger$-move (when $k\geq3$) in the second row of \autoref{A1}. Thus, $\mathcal{K}_1$ and $\mathcal{K}_2$ represent diffeomorphic $n$-manifolds. Here, $\mathcal{K}_2$ is a cancelling $(k-1,k)$-pair, so it represents $B^n$; see \autoref{A3}.
        \end{enumerate}
        \item Consider $\mathcal{K}_3$ and $\mathcal{K}_4$.
        \begin{enumerate}
            \item When $k=2$ and $n\geq4$, the fundamental groups of the induced $n$-manifolds are $\pi_1(M_{\mathcal{K}_3})=\langle x_1, x_2\mid x_1x_2x_1x_2^{-1}x_1^{-1}x_2^{-1}\rangle\ncong\mathbb{Z}$ and $\pi_1(M_{\mathcal{K}_4})=\langle x_1,x_2\mid x_1x_2^{-1}\rangle\cong\mathbb{Z}$, so $\mathcal{K}_3$ and $\mathcal{K}_4$ represent non-diffeomorphic $n$-manifolds. In $\mathcal{K}_4$, we can remove a cancelling $(k-1,k)$-pair, so $\mathcal{K}_4$ represents $S^{k-1}\times B^{n-k+1}$. 
            \item When $k\geq3$ and $n\geq2k+1$, $\mathcal{K}_3$ and $\mathcal{K}_4$ are related by  $\ddagger$-moves, so they represents diffeomorphic $n$-manifolds.
        \end{enumerate}
        \item Consider $\mathcal{K}_5$ and $\mathcal{K}_6$.
        \begin{enumerate}
            \item When $k=2$ and $n\geq4$, the fundamental groups of the induced $n$-manifolds are $\pi_1(M_{\mathcal{K}_5})=\langle x_1,x_2\mid x_1x_2x_1^{-1}x_2^{-1}\rangle\cong\mathbb{Z}$ and $\pi_1(M_{\mathcal{K}_6})=\langle x_1\rangle*\langle x_2\rangle\cong\mathbb{Z}*\mathbb{Z}$, so they represent non-diffeomorphic $n$-manifolds.
            \item When $k\geq3$ and $n\geq2k+1$, $\mathcal{K}_5$ and $\mathcal{K}_6$ are related by a $\ddagger$-move, so they represent diffeomorphic $n$-manifolds. In particular, $M_{\mathcal{K}_6}\cong \natural^2(S^{k-1}\times B^{n-k+1})\natural (S^k\times B^{n-k})$.
        \end{enumerate}
    \end{enumerate}
\end{example}

\begin{figure}[ht!]
 \labellist
 \small\hair 2pt
 \pinlabel {\large{$\mathcal{K}_1$}}  at 115 440
 \pinlabel {\large{$\mathcal{K}_2$}}  at 375 440
 \pinlabel {\large{$\mathcal{K}_3$}}  at 115 250
 \pinlabel {\large{$\mathcal{K}_4$}}  at 375 250
 \pinlabel {\large{$\mathcal{K}_5$}}  at 115 30
 \pinlabel {\large{$\mathcal{K}_6$}}  at 375 30

 \pinlabel{$0$} at 50 480
 \pinlabel{$0$} at 330 500

 \pinlabel{$0$} at 115 295
 \pinlabel{$0$} at 375 320

 \pinlabel{$0$} at 160 130
 \pinlabel{$0$} at 450 130

 \endlabellist
 \centering
 
 \includegraphics[width=.68\textwidth]{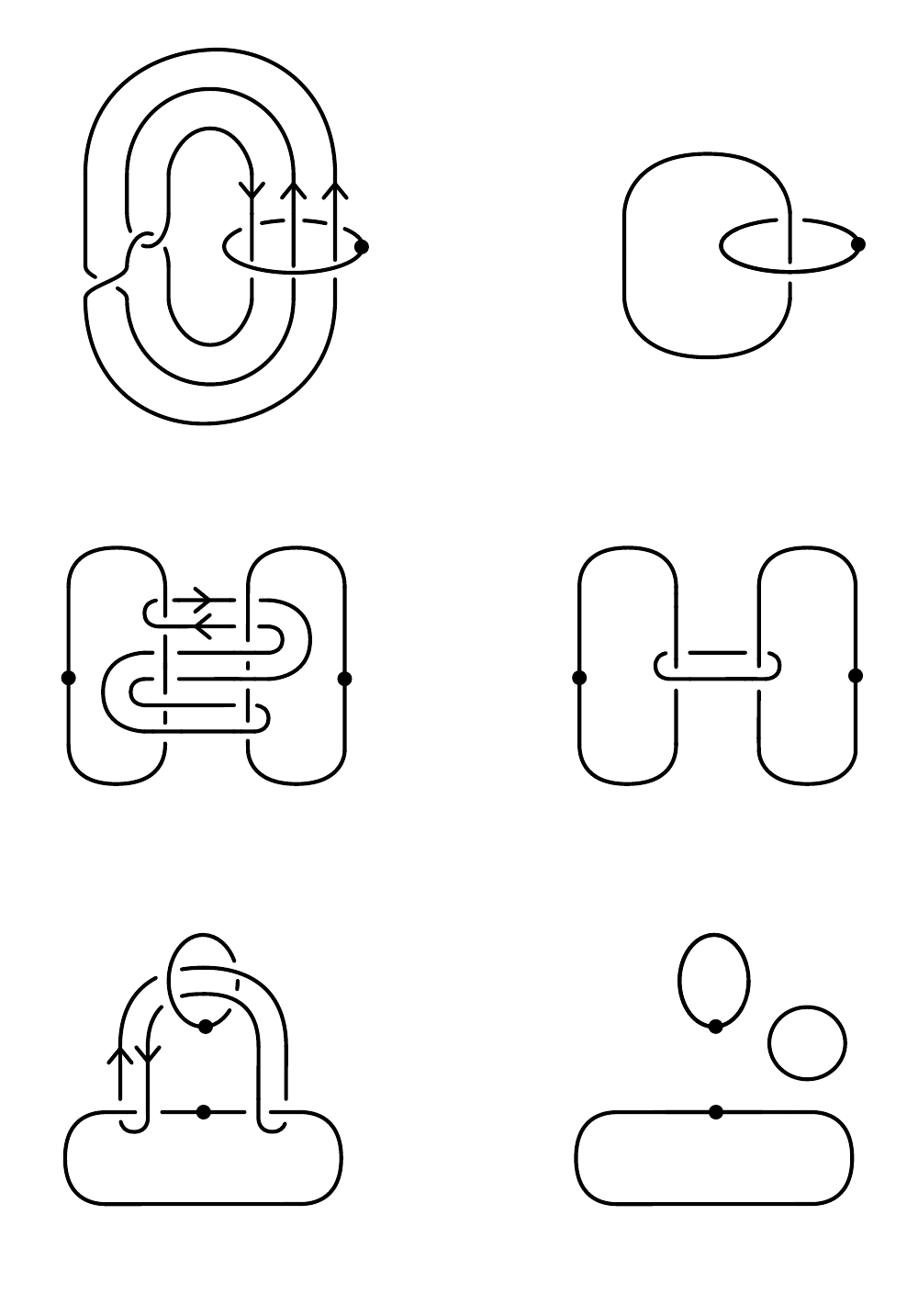}
 \caption{Some $(n,k)$-Kirby diagrams.
 \label{A2}}
\end{figure}
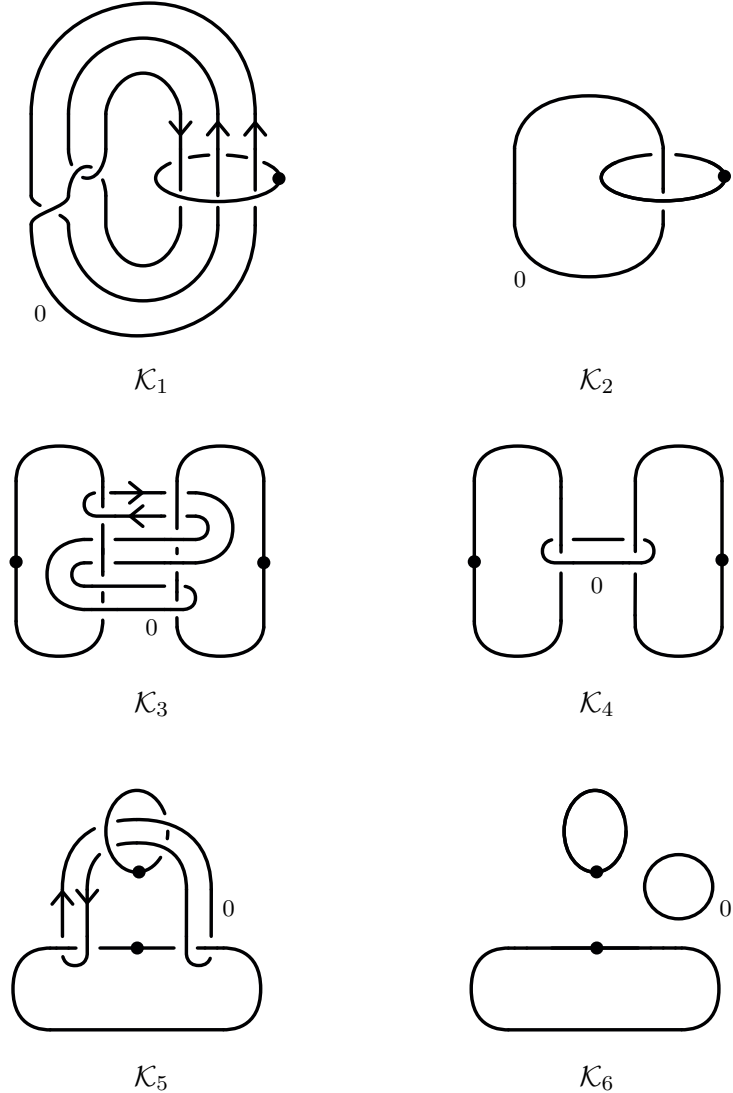

\begin{remark}
 Let $\mathcal{K}$ and $\mathcal{K}'$ be $(n,k)$-Kirby diagrams, where $k\geq2$ and $n\geq2k+1$. We can visualize isotopies of $(n,k)$-Kirby diagrams, $(k-1)$-handle slides, $k$-handle slides, and the creation/annihilation of cancelling $(k-1,k)$-pair on their $(4,2)$-Kirby diagrams; see \autoref{A3}. For some special isotopies of a $(k-1)$-link, which are induced by crossing change and $\ddagger$-move, see \autoref{A1}. By \cite{cerf1970stratification}, if two $(n,k)$-Kirby diagrams, $\mathcal{K}$ and $\mathcal{K}'$ are related by these moves, then the two induced $n$-dimensional $k$-handlebodies, $M_\mathcal{K}$ and $M_{\mathcal{K}'}$, are diffeomorphic.
   
\end{remark}

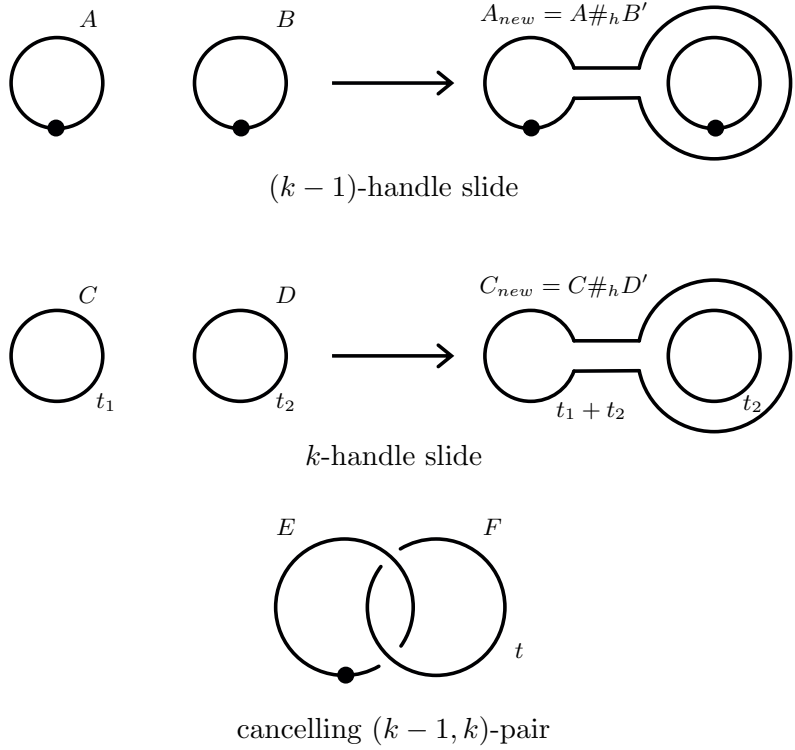
\begin{figure}[ht!]
 \labellist
 \small\hair 2pt
 \pinlabel {\large{$(k-1)$-handle slide}}  at 230 320

 \pinlabel {\large{$k$-handle slide}}  at 230 170

 \pinlabel {\large{cancelling $(k-1,k)$-pair}}  at 230 15

 \pinlabel {$t_1$}  at 70 200
 \pinlabel {$t_2$}  at 170 200
 \pinlabel {$t_1+t_2$}  at 340 195
 \pinlabel {$t_2$}  at 430 199

 \pinlabel {$t$} at 300 60

 \pinlabel {$A$}  at 60 415
 \pinlabel {$B$}  at 170 415
 \pinlabel {$A_{new}=A\#_hB'$} at 325 418

 \pinlabel {$C$}  at 60 260
 \pinlabel {$D$}  at 170 260
 \pinlabel {$C_{new}=C\#_hD'$} at 325 265

 \pinlabel {$E$}  at 170 130
 \pinlabel {$F$}  at 285 130
 
 \endlabellist
 \centering
 
 \includegraphics[width=.7\textwidth]{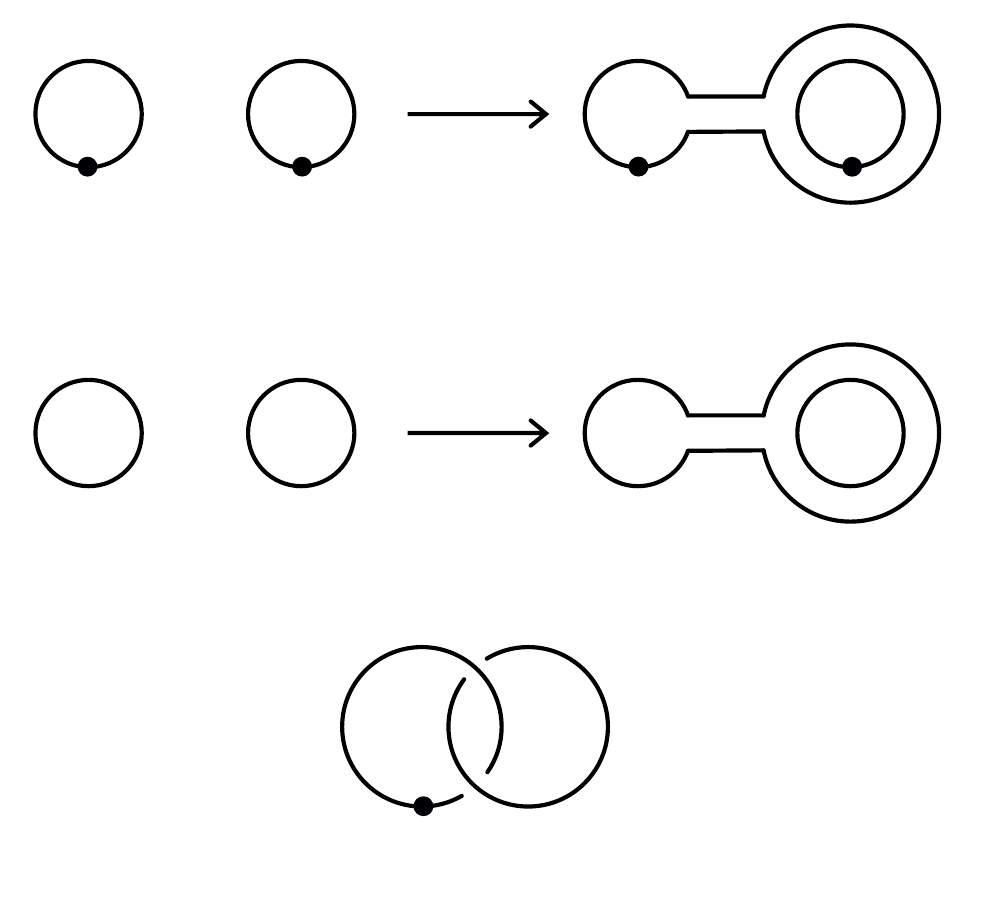}
 \caption{Moves defined on $(n,k)$-Kirby diagrams, where $k\geq2$ and $n\geq2k+1$. A dotted circle in $\{0\}\times\mathbb{R}^3\times\{0\}$ represents a dotted trivial $(n-k-1)$-sphere in $\mathbb{R}^{n-k-2}\times\mathbb{R}^3\times\mathbb{R}^{k-2}$. A circle with a number $t$ in $\{0\}\times\mathbb{R}^3\times\{0\}$ represent a $t$-framed trivial $(k-1)$-sphere in $\mathbb{R}^{n-k-2}\times\mathbb{R}^3\times\mathbb{R}^{k-2}$, where $t\in G$. \textbf{First row}: A $(k-1)$-handle slide of $A$ over $B$ is a dotted $(n-k-1)$-sphere $A_{new}=A\#_hB'$ obtained from $A\cup B'$ by surgery along an $(n-k)$-dimensional $1$-handle $h=B^1\times B^{n-k-1}$ attached to $A$ and a parallel copy $B'$ of $B$. Here, we can always isotope the core $B^1\times\{0\}$ of the $1$-handle $B^1\times B^{n-k-1}$ into $\{0\}\times\mathbb{R}^3\times\{0\}\subset \mathbb{R}^{n-k-2}\times\mathbb{R}^3\times\mathbb{R}^{k-2}$, because homotopy implies isotopy for $1$-manifolds in $\mathbb{R}^{n-1}$. In simple terms, $A_{new}$ is obtained from $A\cup B'$ by surgery along a visible $2$-dimensional band $B^1\times B^1\times\{0\}$ of the $1$-handle $B^1\times B^1\times B^{n-k-2}=B^1\times B^{n-k-1}$. \textbf{Second row}: A $k$-handle slide of $t_1$-framed $C$ over $t_2$-framed $D$ is a $(t_1+t_2)$-framed $(k-1)$-sphere $C_{new}=C\#_hD'$ obtained from $C\cup D'$ by surgery along a $k$-dimensional $1$-handle $h=B^1\times B^{k-1}$ attached to $C$ and a parallel copy $D'$ of $D$. Note that there is no linking between a framed $(k-1)$-sphere $D$ and its parallel copy $D'$ when $k\geq2$ and $n\geq2k+1$. Similarly, the core of the $1$-handle can be isotoped into $\{0\}\times\mathbb{R}^3\times\{0\}\subset \mathbb{R}^{n-k-2}\times\mathbb{R}^3\times\mathbb{R}^{k-2}$, and $C_{new}$ is obtained from $C\cup D'$ by surgery along a visible band $B^1\times B^1\times\{0\}$ of the $1$-handle $B^1\times B^1\times B^{k-2}=B^1\times B^{k-1}$. \textbf{Third row}: A cancelling $(k-1,k)$-pair is the union of a dotted $(n-k-1)$-sphere $E$ and a framed $(k-1)$-sphere $F$, with the geometric intersection number is $|D_{E}\pitchfork F|=1$, where $D_{E}$ is the trivial $(n-k)$-ball of $E$ and $E\cap (\{0\}\times\mathbb{R}^3\times\{0\})$ is a visible $2$-disk bounded by the dotted circle.
 \label{A3}}
\end{figure}

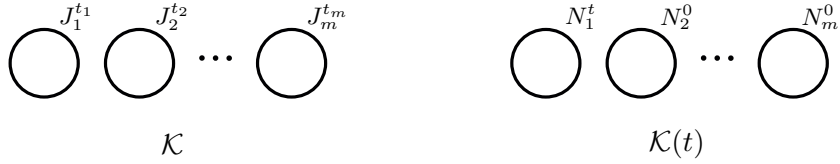
\begin{figure}[ht!]
 \labellist
 \small\hair 2pt

 \pinlabel {$J^{t_1}_1$}  at 50 110
 \pinlabel {$J^{t_2}_2$}  at 110 110
 \pinlabel {$J^{t_m}_m$}  at 207 110
 \pinlabel {\large{$\mathcal{K}$}} at 110 30

 \pinlabel {$N^{t}_1$}  at 365 110
 \pinlabel {$N^{0}_2$}  at 425 110
 \pinlabel {$N^{0}_m$}  at 515 110
 \pinlabel {\large{$\mathcal{K}(t)$}} at 425 30
 
 \endlabellist
 \centering
 
 \includegraphics[width=0.7\textwidth]{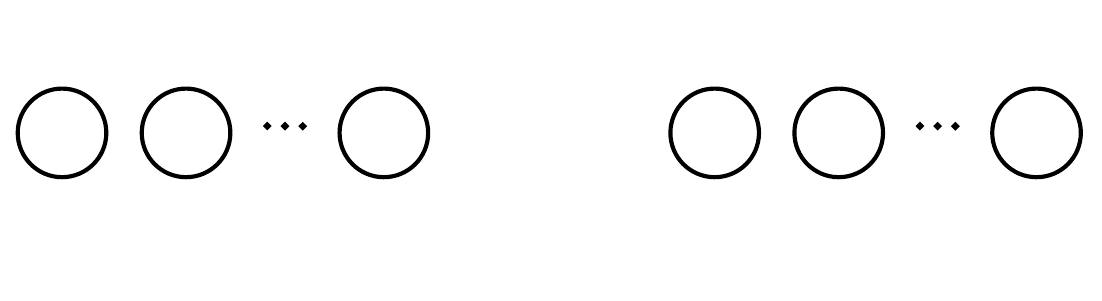}
 \caption{$(n,k)$-Kirby diagrams $\mathcal{K}$ and $\mathcal{K}(t)$. \label{A4}}
\end{figure}

\begin{theorem}\label{thm: classification of simple handlebodies}
    Fix $k\geq2$ and $n\geq2k+1$. Let $X$ be an $n$-dimensional $k$-handlebody with a $0$-handle and $m$ $k$-handles. Then $X$ is diffeomorphic to $M_{\mathcal{K}(t)}$ for some $(n,k)$-Kirby diagram $\mathcal{K}(t)$ in the right of \autoref{A4}. In particular, $M_{\mathcal{K}(t)}$ is the boundary connected sum of a $B^{n-k}$-bundle over $S^{k}$ and the trivial bundle $\natural^{m-1}(S^k\times B^{n-k})$.
\end{theorem}
\begin{proof}
    Let $\mathcal{K}=L_1\cup L_2\subset S^{n-1}$ be an $(n,k)$-Kirby diagram of $X$. Since $L_1=\emptyset$, the diagram $\mathcal{K}$ consists solely of the trivial $(k-1)$-link. Thus, we can assume that \[\mathcal{K}=J^{t_1}_1\cup\cdots\cup J^{t_m}_m\subset S^{n-1},\] where $J^{t_i}_i$ is a $t_i$-framed $(k-1)$-knot and $t_i\in G$, as depicted in the left of \autoref{A4}. 
    
    Let \[t=\text{gcd}(t_1,\dots,t_m)\] be the greatest common divisor of the framings. By performing handle slides among the components of $\mathcal{K}$, we can rearrange them so that only one component of $\mathcal{K}$ remains a $t$-framed $(k-1)$-knot, while all other components have framing $0$. More precisely, the resulting $(n,k)$-Kirby diagram is \[\mathcal{K}(t)=N^t_1\cup N^0_{2}\cup\cdots\cup N^0_{m},\] as shown in the right of \autoref{A4}.
    
    If necessary, we can further resolve crossings by performing crossing changes, as described in \autoref{A1}.
    Therefore, $X$ is diffeomorphic to $M_{\mathcal{K}(t)}$, which, in turn, is diffeomorphic to the boundary connected sum of a $B^{n-k}$-bundle over $S^{k}$ and the trivial bundle $\natural^{m-1}(S^k\times B^{n-k})$.
\end{proof}

\begin{figure}[ht!]
 \labellist
 \small\hair 2pt
 \pinlabel {$p_1$}  at 145 240
 \pinlabel {$p_m$}  at 145 120
 \pinlabel {$L_1$}  at 50 100
 \pinlabel {$J^{t_1}_1$}  at 210 285
 \pinlabel {$J^{t_m}_m$}  at 210 170
 \pinlabel {\large{$\mathcal{K}=L_1\cup L_2$}} at 140 30
 
 \pinlabel {$p$}  at 363 172
 \pinlabel {$L_1$}  at 298 100
 \pinlabel {$N^{a}_1$}  at 418 250
 \pinlabel {$N^{s_2}_2$}  at 518 300
 \pinlabel {$N^{s_3}_3$}  at 518 230
 \pinlabel {$N^{s_m}_m$}  at 518 120
 \pinlabel {\large{$\mathcal{K}'=L_1\cup L_2'$}} at 410 30

 \pinlabel {$p$}  at 665 172
 \pinlabel {$L_1$}  at 600 100
 \pinlabel {$N^{a}_1$}  at 720 250
 \pinlabel {$E^{b}_2$}  at 820 300
 \pinlabel {$E^{0}_3$}  at 820 230
 \pinlabel {$E^{0}_m$}  at 820 120
 \pinlabel {\large{$\mathcal{K}(p;a,b)=L_1\cup L_2''$}} at 700 30

 \endlabellist
 \centering
 
 \includegraphics[width=0.9\textwidth]{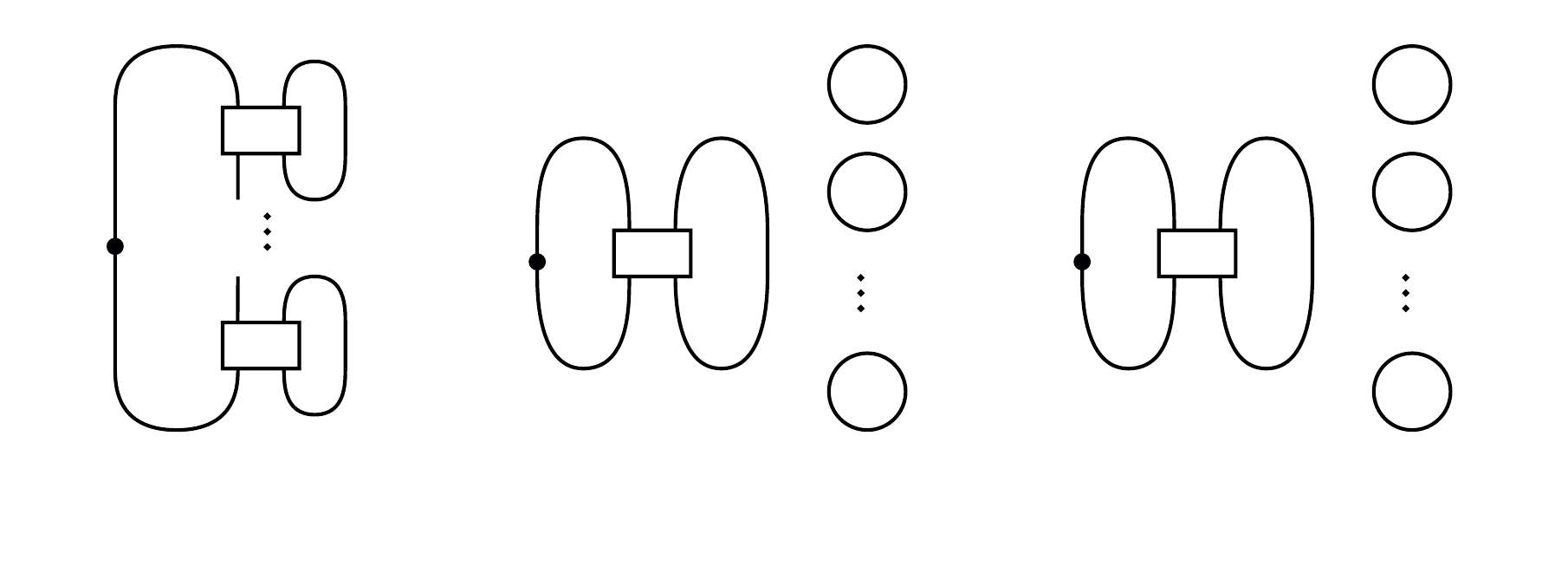}
 \caption{$(n,k)$-Kirby diagrams $\mathcal{K},\mathcal{K}'$, and $\mathcal{K}(p;a,b)$. \textbf{Left}: The diagram $\mathcal{K}=L_1\cup L_2\subset S^{n-1},$ where $L_1$ is a dotted $(n-k-1)$-sphere, and $L_2=J^{t_1}_{1}\cup \cdots \cup J^{t_m}_m$ is a framed $(k-1)$-link consisting of  $t_i$-framed $(k-1)$-knots $J^{t_i}_i$. The components $L_1$ and $L_2$ are linked, while $L_2$ itself forms the trivial $(k-1)$-link. Each $p_i$ denotes $p_i$-full twist. \textbf{Middle}: The modified diagram $\mathcal{K}'=L_1\cup L_2'\subset S^{n-1},$ where $L_2'=N^{a}_{1}\cup N^{s_2}_2\cup \cdots \cup N^{s_m}_m$ is a framed $(k-1)$-link with $N^{s_i}_i$ being an $s_i$-framed $(k-1)$-knot and $N^{a}_1$ an $a$-framed $(k-1)$-knot. 
 The components $L_1$ and $N^{a}_1$ are linked. Here, $p$ presents a $p$-full twist. \textbf{Right}: The simplified diagram $\mathcal{K}(p;a,b)=L_1\cup L_2''$, where $L_2''=N^{a}_{1}\cup E^{b}_2\cup E^{u_3}_3\cup \cdots \cup E^{u_m}_m$ is a framed $(k-1)$-link, with $u_3=\cdots=u_m=0$. \label{A5}}
\end{figure}
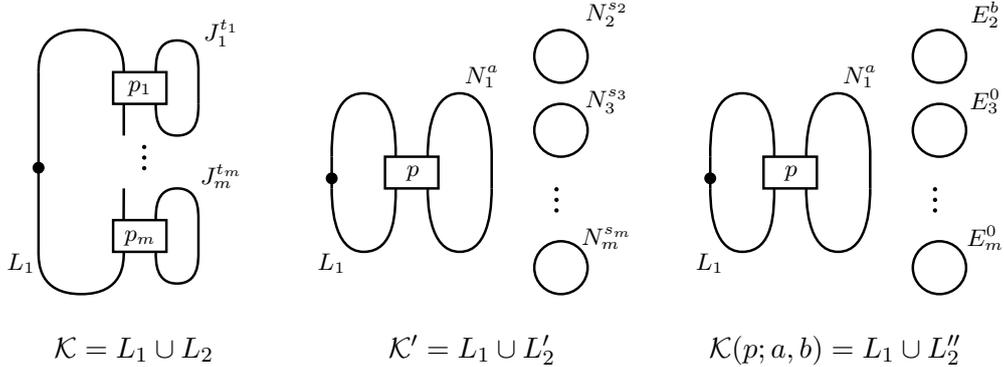

\classificationofcertainhandlebodies

\begin{proof}
    Let $\mathcal{K}=L_1\cup L_2\subset S^{n-1}$ be an $(n,k)$-Kirby diagram of $X$. Since $L_2$ is the trivial $(k-1)$-link, we can assume that $\mathcal {K}$ is the diagram in the left of \autoref{A5} after applying crossing changes and $\ddagger$-moves in \autoref{A1}, if necessary. 
    
    Here, $L_2=J_1^{t_1}\cup\cdots\cup J_m^{t_m}$, where $J_i^{t_i}$ is a $t_i$-framed $(k-1)$-knot, and $p_i$ indicates $p_i$-full twist. Let \[p=\text{gcd}(p_1,\dots,p_m)\] be the greatest common divisor of the full twists. By performing handle slides among the components of $L_2$, we can arrange them so that only one component of $L_2$ remains linked with $L_1$. That is, the resulting $(n,k)$-Kirby diagram is \[\mathcal{K}'=L_1\cup L_2',\] where \[L_2'=N_1^{a}\cup N^{s_2}_2\cup \cdots\cup N_m^{s_m}\] as shown in the middle of \autoref{A5}, with $a,s_i\in G$. 
    
    Next, we further perform handle slides among the components of $L_2'\setminus N_1=N_2\cup\cdots\cup N_m$ so that only one component has a framing \[b=\text{gcd}(s_2,\dots,s_m),\] while the other components have framings $0$. This results in the $(n,k)$-Kirby diagram \[\mathcal{K}(p;a,b)=L_1\cup L_2'',\] where \[L_2''=N_1^{a}\cup E_2^{b}\cup E^0_3\cup\cdots\cup E_m^{0}\] as depicted in the right of \autoref{A5}. Thus, $X$ is diffeomorphic to $M_{\mathcal{K}(p;a,b)}$.
\end{proof}

\begin{remark} See the right of \autoref{A5}.
\begin{enumerate}
    \item When $p=0$, we can perform handle slides between $N_1$ and $N_2$ until the framing of $E_2$ becomes $0$. Therefore, we can assume that $b=0$ in this case.
    \item When $p=1$, we can delete the cancelling $(k-1,k)$-pair $L_1\cup N_1$.
    \item The $(k-1)$-th homotopy group of $M_{\mathcal{K}(p;a,b)}$ is
    \begin{equation*}
\pi_{k-1}(M_{\mathcal{K}(p;a,b)})\cong
    \begin{cases}
        \mathbb{Z}&\hspace{5mm}\text{if}\;\; p=0 \\
        \{0\} & \hspace{5mm}\text{if}\;\; p=1  \\
        \mathbb{Z}_p & \hspace{5mm}\text{if}\;\; p\geq2.
    \end{cases}
\end{equation*}
    \item When $k=3,5,6$ or $7$ $(\text{mod}\;8)$, there is a unique framing of each $(k-1)$-sphere. Therefore, $\mathcal{K}(p;a,b)$ depends only on $p$.
\end{enumerate}
\end{remark}

\begin{remark}\label{rem: relation between (4,2) and (n,2)}
Let $X$ be an $n$-dimensional $2$-handlebody, where $n\geq5.$ By \autoref{thm: product structure on handlebodies}, there exists a $4$-dimensional $2$-handlebody $Y$ such that $X=Y\times B^{n-4}$.

Let $\mathcal{K}=L_1\cup L_2\subset S^3$ be a Kirby diagram (or $(4,2)$-Kirby diagram) of $Y$. Then the induced $(n,2)$-Kirby diagram \[\tilde{\mathcal{K}}=\tilde{L_1}\cup \tilde{L_2}\subset S^{n-1}\] is an $(n,2)$-Kirby diagram of $X$. 

Let $J\subset L_2$ be an $m$-framed $1$-knot, a component of $L_2$, where $m\in \mathbb{Z}$. An $m$-framed $1$-knot $J$ in $S^3$ is a pair $(J,\phi)$, where \[\phi:S^1\times B^2\hookrightarrow S^3\] is an embedding such that
\begin{enumerate}
    \item $\phi(S^1\times\{(0,0)\})=J$,
    \item $\phi(S^1\times B^{2})=\nu(J)$,
    \item the linking number between $J$ and $J'=\phi(S^1\times\{(1,0)\})$ is $\text{lk}(J,J')=m$.
\end{enumerate} 

Note that the induced homomorphism \[i_{*}:\mathbb{Z}\cong\pi_1(O(2))\rightarrow\pi_1(O(n-2))\cong\mathbb{Z}_2\] is an epimorphism, where $i$ is the natural inclusion of orthogonal groups; see \autoref{lem:induced map}. Then the induced $1$-knot $\tilde{J}\subset\tilde{L_1}$ is $u$-framed, where $u\in\{0,1\}=\mathbb{Z}_2$ and $u\equiv m$ (mod $2$); see also \autoref{rmk: induced framings}. 

Now, the induced $n$-manifold $M_{\tilde{\mathcal{K}}}$ is diffeomorphic to $M_{\mathcal{K}}\times B^{n-4}$, allowing us to interpret the $n$-manifolds represented by the diagrams in \autoref{A6}. 

For example, the fifth diagram in \autoref{A6} represents $L(p,1)^{\circ}\times B^1$ as a $(4,2)$-Kirby diagram, so it represents $L(p,1)^{\circ}\times B^{n-3}$ as an $(n,2)$-Kirby diagram, where $L(p,1)^{\circ}$ is the punctured lens space $L(p,1)$. The third diagram in \autoref{A6} represents the punctured $\mathbb{C}P^2$ as a $(4,2)$-Kirby diagram, so it represents the non-trivial $B^{n-2}$-bundle over $S^2$, denoted by $S^2\tilde{\times}B^{n-2}$. Note that there are only two possible $B^{n-2}$-bundles over $S^2$; one is $S^2\times B^{n-2}$ and the other is $S^2\tilde{\times} B^{n-2}$. This classification follows from the fact that $\pi_1(O(n-2))\cong\mathbb{Z}_2$ when $k\geq2$ and $n\geq2k+1$. 
\end{remark}

\begin{definition}
    Let $\mathcal{K}=L_1\cup L_2$ and $\mathcal{K}'=L_1\cup L_2'$ be $(4,2)$-Kirby diagrams, where the two framed links \[L_2=J_1^{t_1}\cup\cdots\cup J_m^{t_m}\quad \text{and} \quad L_2'=J_1^{s_1}\cup\cdots\cup J_m^{s_m}\] equal as links. If $s_i\equiv t_i$ (mod 2) for every $i\in\{0,\dots,m\}$, we say that $\mathcal{K}$ and $\mathcal{K}'$ are \textit{weakly equivalent}.
\end{definition}

\begin{example}
    Let $\mathcal{K}\subset S^3$ be the $(-1)$-framed trefoil knot and $\mathcal{K'}\subset S^3$ be the $2025$-framed trefoil knot. Then they are weakly equivalent. Clearly, the induced $(n,2)$-Kirby diagrams $\tilde{\mathcal{K}}$ and $\tilde{\mathcal{K}'}$ represent diffeomorphic $n$-manifolds because $-1\equiv2025$ (mod 2).
\end{example}

\begin{remark}
    Let $X$ and $X'$ be $n$-dimensional $2$-handlebodies, where $n\geq 5$. By \autoref{thm: product structure on handlebodies}, there exist $4$-dimensional $2$-handlebodies $Y$ and $Y'$ such that \[X=Y\times B^{n-4}\quad \text{and} \quad X'=Y'\times B^{n-4}.\] Let $\mathcal{K}$ and $\mathcal{K}'$ be $(4,2)$-Kirby diagrams of $Y$ and $Y'$, respectively. Then the induced $(n,2)$-Kirby diagrams $\tilde{\mathcal{K}}$ and $\tilde{\mathcal{K}'}$ are $(n,2)$-Kirby diagrams of $X$ and $X'$.
    
    If $\mathcal{K}$ ane $\mathcal{K}'$ are related by the usual Kirby moves (isotopies, handle slides, and annihilation/creation of cancelling pair), by crossing changes as in \autoref{A1}, or if they are weakly equivalent, then the induced $n$-dimensional $k$-handlebodies $M_{\tilde{\mathcal{K}}}$ and $M_{\tilde{\mathcal{K}'}}$ are diffeomorphic.
\end{remark}

\begin{example}\hfill
\begin{enumerate}
    \item The $(-1)$-framed trefoil knot and the third diagram in \autoref{A6} are weakly equivalent after performing a crossing change on the $(-1)$-framed trefoil knot.
    \item The seventh diagram in \autoref{A6} is obtained from the sixth diagram by handle sliding $0$-framed component over the other component and performing a crossing change.
\end{enumerate}
\end{example}

\begin{figure}[ht!]
 \labellist
 \small\hair 2pt
 \pinlabel {\large{$S^1\times B^{n-1}$}}  at 75 480
 \pinlabel {\large{$S^2\times B^{n-2}$}}  at 235 480
 \pinlabel {\large{$S^2\tilde{\times} B^{n-2}$}}  at 398 480
 \pinlabel {\large{$B^n$}}  at 110 250
 \pinlabel {\large{$L(p,1)^{\circ}\times B^{n-3}$}}  at 365 250
 \pinlabel {\large{$(S^2\times B^{n-2})\natural (S^2\tilde{\times}B^{n-2})$}}  at 110 50
 \pinlabel {\large{$(S^2\tilde{\times} B^{n-2})\natural (S^2\tilde{\times}B^{n-2})$}}  at 365 50
 
 \pinlabel {$0$}  at 277 580
 \pinlabel {$1$}  at 440 580
 \pinlabel {$t\in\{0,1\}$}  at 180 390
 \pinlabel {$p$}  at 363 340
 \pinlabel {$0$}  at 430 390
 \pinlabel {$0$}  at 90 160
 \pinlabel {$1$}  at 195 160
 \pinlabel {$1$}  at 345 160
 \pinlabel {$1$}  at 450 160

 \endlabellist
 \centering
 
 \includegraphics[width=.7\textwidth]{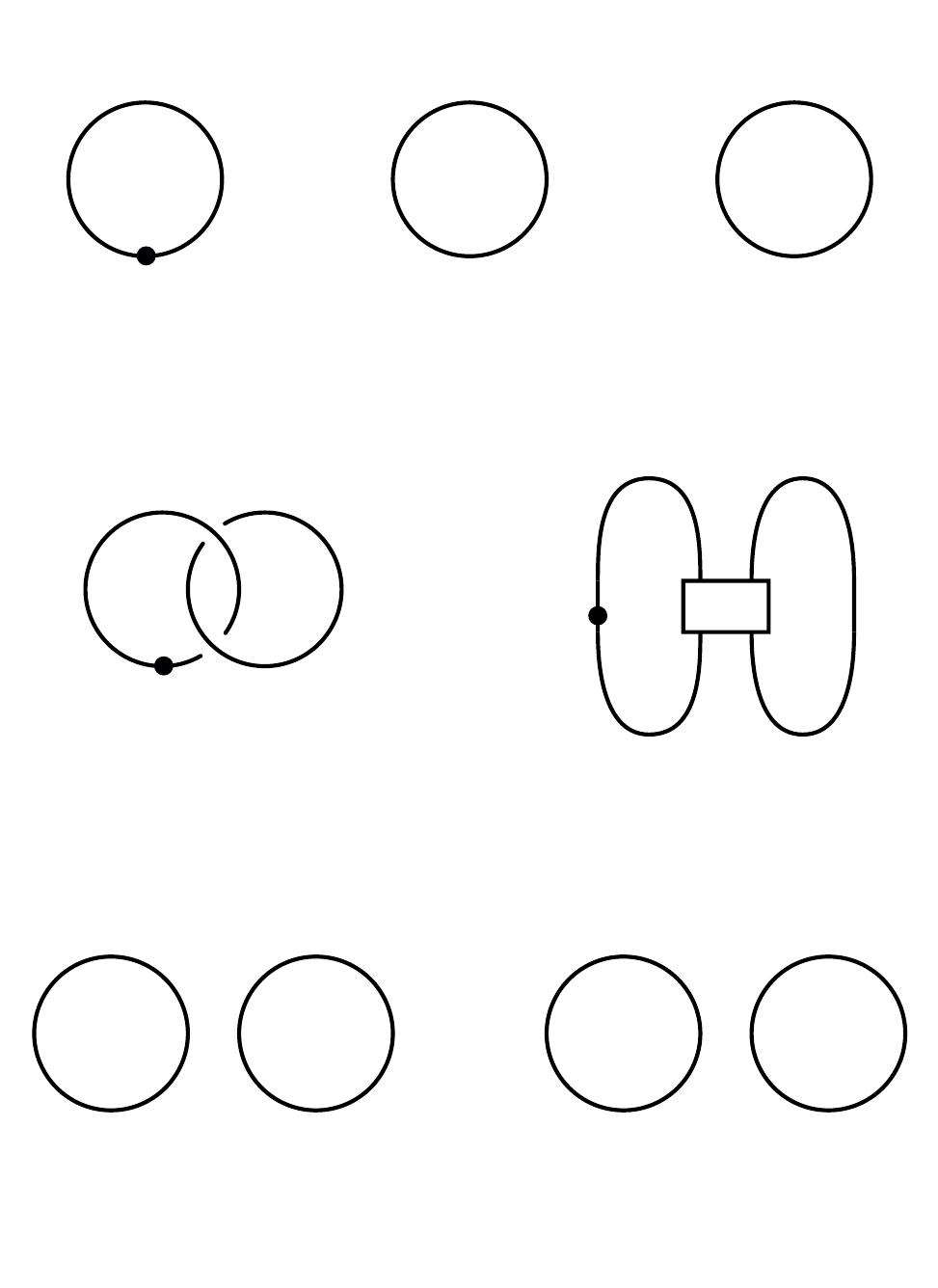}
 \caption{$(n,2)$-Kirby diagrams of some $n$-dimensional $2$-handlebodies. \label{A6}}
\end{figure}
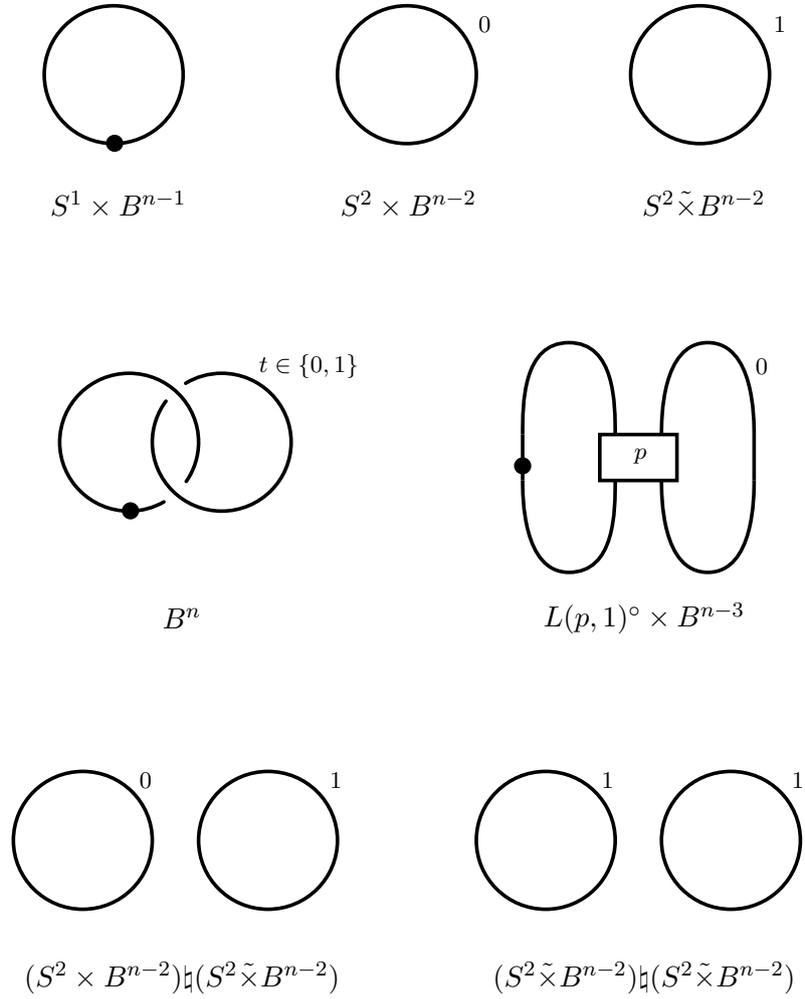

\bibliographystyle{alpha} 
\bibliography{refs} 

\end{document}